\documentclass[10pt]{article}

\RequirePackage{amsthm,amsmath,amsfonts,amssymb}
\RequirePackage[numbers]{natbib}
\RequirePackage[colorlinks,citecolor=blue,urlcolor=blue]{hyperref}
\RequirePackage{graphicx}
\RequirePackage{tikz}
	\usetikzlibrary{calc}
\RequirePackage{enumitem}
\RequirePackage[affil-it]{authblk}
\RequirePackage[a4paper]{geometry}

\theoremstyle{plain}
\newtheorem{theorem}{Theorem}
\newtheorem*{theorem*}{Theorem}
\newtheorem*{corollary*}{Corollary}
\newtheorem{lemma}[theorem]{Lemma}
\newtheorem{proposition}[theorem]{Proposition}
\newtheorem{corollary}[theorem]{Corollary}
\newtheorem{conjecture}[theorem]{Conjecture}
\theoremstyle{remark}
\newtheorem*{remark}{Remark}
\newtheorem{definition}[theorem]{Definition}
\newtheorem*{definition*}{Definition}
\newtheorem*{claim}{Claim}

\def\R{\mathbb{R}}
\def\N{\mathbb{N}}
\def\Z{\mathbb{Z}}

\def\C{\mathbb{C}}
\def\P{\mathbb{P}}
\def\E{\mathbb{E}}

\def\t#1{\widetilde{#1}}
\def\1{\mathbbold{1}}
\def\eps{\varepsilon}
\def\d{\mathrm{d}}

\def\rdown#1{\left\lfloor #1 \right\rfloor}
\def\D{\mathbb{D}}
\def\H{\mathbb{H}}
\def\T{\mathbb{T}}

\def\cparam{\mathbf{c}}

\def\cudisc{\overline{\mathbb{D}}} 
\def\deadtime{\tau_D}
\def\etasize{\nu}
\def\thetafake{\theta^{\bot}}
\def\thetalegit{\theta^{\top}}
\def\discoeff{2(e^{\cparam} - 1)^{\frac{1}{4}}} 
\def\bp{\widehat{z}}
\def\cadlag{c\`adl\`ag }
\def\sigmaub{\cparam^{2^{2^{1/\cparam}}}}
\def\Lj#1{\frac{\beta}{4}\left( \frac{L}{\beta} \right)^{2^{#1}}}

\title{SLE scaling limits for a Laplacian growth model}
\author{Frankie Higgs \\ Department of Mathematics and Statistics, Lancaster University, 
Lancaster LA1 4YF, UK.
\thanks{f.higgs@lancaster.ac.uk}
}

\begin{document}

\maketitle

\begin{abstract}
	We consider a model of planar random aggregation from the ALE$(0,\eta)$ family
	where particles are attached preferentially in areas of low harmonic measure.
	We find that the model undergoes a phase transition in negative $\eta$,
	where for sufficiently large values
	the attachment distribution of each particle becomes atomic in the small particle limit,
	with each particle attaching to one of the two points at the base of the
	previous particle.
	This complements the result of Sola, Turner and Viklund
	for large positive $\eta$,
	where the attachment distribution condenses to a single atom
	at the tip of the previous particle.
	
	As a result of this condensation of the attachment distributions
	we deduce that
	in the limit as the particle size tends to zero
	the ALE cluster converges to a Schramm--Loewner evolution
	with parameter $\kappa = 4$ (SLE$_4$).
	
	We also conjecture that using other particle shapes
	from a certain family, we have a similar SLE scaling result,
	and can obtain SLE$_\kappa$ for any $\kappa \geq 4$.
\end{abstract}


\tableofcontents

\section{Introduction}

\subsection{Conformal aggregation}
\label{sec:conformal-aggregation}

There has been a great deal of research into models
of random aggregation, where particles are added at each time step
to the existing cluster at random locations.
These models are perhaps most easily defined on the lattice $\Z^d$,
where each particle is one vertex, for example
diffusion-limited aggregation (DLA) \cite{witten1981diffusion}
or the Eden model \cite{eden1961two}.
However, the underlying anisotropy of $\Z^d$
may be retained by the cluster on large scales,
making these models a poor approximation of reality
under some conditions \cite{dlasim-gb} \cite{eden-high-dim}.

In two dimensions we may change
to a setting without this problem;
models of \emph{conformal growth}
existing in the complex plane $\C$
rather than $\Z^2$.
In this paper we will study
the aggregate Loewner evolution (ALE($\alpha,\eta$))
model introduced in \cite{stv-ale},
which is a generalisation of
the Hastings--Levitov model (HL($\alpha$)) \cite{hl-model}.\\

In a conformal aggregation model,
we add particles to our cluster by composing conformal maps
from a fixed reference domain to smaller domains.
Our initial cluster will be the closed unit disc
$K_0 = \overline{\D} = \{ z \in \C : |z| \leq 1 \}$.
We attach a particle to $K_0$ by applying a map
from
its complement in the Riemann sphere $\C_\infty$,
$\Delta := \C_\infty \setminus \overline{\D}$,
to a smaller domain, and then the new cluster will
be the complement of the image of $\Delta$.
We will use particles of the form
$(1, 1+d]$ for $d > 0$.
%
%
\begin{definition}
For any $d > 0$,
by the Riemann mapping theorem
there exists a unique bijective conformal map
\[
 f^d \colon \Delta \to \Delta \setminus (1, 1+d]
\]
such that $f^d(z) = e^{\cparam}z + O(1)$ near $\infty$,
for some $\cparam = \cparam(d) \in \R$.
\end{definition}
One advantage of the slit particles
we use in this paper
over more general particle shapes
is that
we have an explicit expression
for $f^d(z)$ \cite{marshall-rohde}.
We call $\cparam > 0$
the \emph{(logarithmic) capacity} of the particle.
As the name suggests, we can view $\cparam$
as measuring the ``size'' of a set in a certain sense.
As we consider the ``small-particle limit''
we will parameterise the model
by the particle capacity $\cparam$
(equivalently by $d$).
\begin{definition}
The preimage of the particle $(1,1+d]$ under $f$
is $\{ e^{i\theta} : -\beta \leq \theta \leq \beta \}$,
where $0 < \beta(\cparam) < \pi$
is uniquely determined by $f(e^{i\beta}) = 1$.
\end{definition}
We can explicitly relate the quantities $\cparam$,
$\beta$ and $d$ using two equations
found in \cite{marshall-rohde}
and \cite{stv-ale}:
$4e^{\cparam} = (d+2)^2/(d+1)$ and
$e^{i\beta} = 2e^{-\cparam} - 1 + 2ie^{-\cparam}\sqrt{e^{\cparam} - 1}$.
Asymptotically,
as $\cparam \to 0$,
these give us
$\beta(\cparam) \sim d(\cparam)
\sim 2 \cparam^{1/2}$.\\


We have maps which can attach one particle,
so now we want to be able to build a cluster with multiple particles
by composing maps which attach particles in different positions.
For $\theta \in \R$ and $\cparam > 0$, define the rotated map
\begin{align*}
 &f^{\theta,\cparam} \colon \Delta \to \Delta \setminus e^{i\theta}(1,1+d(\cparam)],\\
 &f^{\theta,\cparam}(z) = e^{i\theta} f^{d(\cparam)}( e^{-i\theta}z ),
\end{align*}
and note that it has the same behaviour $f^{\theta,\cparam}(z) = e^{\cparam} z + O(1)$
near $\infty$ as does $f^{d(\cparam)}$.

Now we want to attach multiple particles.
\begin{definition}
Given a sequence of angles $(\theta_n)_{n \in \N}$
and of capacities $(c_n)_{n \in \N}$,
if we write
$f_j = f^{\theta_j,c_j}$ then
we can define
\begin{equation}
\label{eq:phi-n-def}
 \Phi_n = f_1 \circ f_2 \circ \dotsb \circ f_n,
\end{equation}
and define the $n$th cluster
$K_n$
as the complement of
$\Phi_n(\Delta)$, so
\[
 \Phi_n \colon
 \Delta \to \C_\infty \setminus K_n.
\]
Note that the total capacity is
$\cparam(K_n) = \sum_{k=1}^{n} c_k$,
i.e.\ $\Phi_n(z) = e^{\sum_{k=1}^n c_k}z + O(1)$
near $\infty$.
\end{definition}

We can now use this setup
to construct various models
of \emph{random}
growth,
by choosing the angles
$(\theta_n)_{n \in \N}$
and capacities $(c_n)_{n \in \N}$
according to a stochastic
process.
%

\subsection{Aggregate Loewner evolution}
\label{sec:ale}

The \emph{aggregate Loewner
evolution} model introduced in
\cite{stv-ale}
is a conformal aggregation model
as in
\autoref{sec:conformal-aggregation},
where for the $(n+1)$th particle
the distribution of its attachment angle $\theta_{n+1}$
and its capacity $c_{n+1} = \cparam(P_{n+1})$
are functions of the
density of \emph{harmonic measure}
on the boundary of $K_n$.
The conditional distribution of
$\theta_{n+1}$ and the way
we obtain $c_{n+1}$
are respectively
controlled by the two
parameters $\eta$ and $\alpha$.

\begin{definition}
Inductively,
we choose $\theta_{n+1}$
for $n \geq 0$
conditionally on
$\theta_1, \dotsc, \theta_n$
according to the probability
density function
\begin{equation}
\label{eq:density}
 h_{n+1}(\theta)
 =
 \frac{1}{Z_n}
 \left|
  \Phi_n'\left(
   e^{\sigma + i\theta}
  \right)
 \right|^{-\eta},\,
 \theta \in (-\pi, \pi],
\end{equation}
where
$Z_n 
=
\int_\T
 |
  \Phi_n'(e^{\sigma + i\theta})
 |^{-\eta}
\,\d\theta$
is a normalising factor.
We have introduced a
$\sigma = \sigma(\cparam) > 0$
as the poles and zeroes
of $\Phi_n'$
on the boundary
mean the measure $h_{n+1}$
is not necessarily well-defined
if $\sigma = 0$,
but we take $\sigma \to 0$ as $\cparam \to 0$.

Since $\sigma$ controls the level of
boundary detail captured by $h_{n+1}$,
and in the $\eta < -2$ regime $h_{n+1}$
is concentrated about the \emph{least} prominent
points on the boundary,
in this paper $\sigma$ will decay extremely
quickly.

On the other hand, in \cite{subcritical}
it is shown that
if $\sigma$ does not decay faster
than $\cparam^{1/2}$
as $\cparam \to 0$,
including if $\sigma$
is kept fixed,
the resulting scaling limit is a disc.

By this definition the first attachment point
$\theta_1$ is chosen uniformly on $\T$.
For convenience we work with $\theta_1 = 0$,
and the random case can be recovered by applying
a random rotation to the final cluster.
\end{definition}

After choosing
$\theta_{n+1}$,
we choose the capacity
of the $(n+1)$th particle
to be
\begin{equation}
\label{eq:capacity}
 c_{n+1}
 =
 \cparam
 |
  \Phi_n'(
   e^{\sigma + i\theta_{n+1}}
  )
 |^{-\alpha}
\end{equation}
where $\cparam$ is a 
capacity parameter and
$c_1 = \cparam$,
and we will later consider
the limit shape of the cluster
as $\cparam \to 0$.

\begin{figure}[h]
\def\lss{1} 
 \centering
 \begin{tikzpicture}[rotate=45]
  \draw node (cluster) at (-0.12,-0.13) {
    \includegraphics[trim=0 0 0 0, clip, scale=\lss]{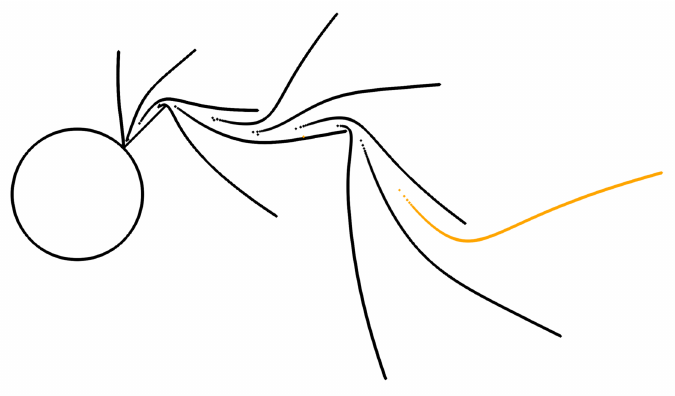}
  };
  \draw[color=red, line width=2pt, dashed]
  (0.8*\lss, -0.7*\lss)
  .. controls
   (0.85*\lss, -0.5*\lss) and (0.9*\lss,0.1*\lss)
  ..
  (0.95*\lss, 0.55*\lss)
  .. controls
   (0.9*\lss, 0.65*\lss)
  ..
  (0.85*\lss, 0.8*\lss)
  .. controls
   (0.6*\lss, 0.85*\lss) and (0.5*\lss, 0.9*\lss)
  ..
  (0.3*\lss, 1*\lss)
  .. controls
   (0*\lss, 1.3*\lss) and (-0.25*\lss, 1.5*\lss)
  ..
  (-0.45*\lss, 1.7*\lss)
  .. controls
   (-0.5*\lss, 1.8*\lss) and (-0.55*\lss, 1.9*\lss)
  ..
  (-0.6*\lss, 2*\lss)
  .. controls
   (-0.7*\lss, 2.3*\lss) and (-0.8*\lss, 2.6*\lss)
  ..
  (-0.9*\lss, 2.9*\lss)
  .. controls
   (-1*\lss, 3.1*\lss)
  ..
  (-1.0*\lss, 3.1*\lss)
  .. controls
   (-1.1*\lss, 3.2*\lss)
  ..
  (-1.25*\lss, 3.15*\lss)
  .. controls
   (-1.5*\lss, 3.1*\lss)
  ..
  (-2*\lss, 3.05*\lss);
 \end{tikzpicture}
 \caption{\label{fig:last-slit}The final particle (the rightmost, in orange)
 of the cluster $K_n$ is highly distorted by the application
 of the first $n-1$ maps $f_{n-1}, f_{n-2}, \dotsc, f_1$.
 The distortion is much greater near the base of the particle:
 we have had to fill in a guess (the red dashed line)
 for the behaviour of the particle deep into the cluster,
 as the distortion is so large there that we are unable to find the exact
 location of enough points to draw a sensible diagram.
 In fact, the red dashed section corresponds to only 1/500\,000th
 the length of the original, undistorted slit.
 }
\end{figure}

\subsection{Our results}
\label{sec:results}

\begin{figure}
 \centering
 \includegraphics[width=1\textwidth, trim= 0 30 0 120, clip]{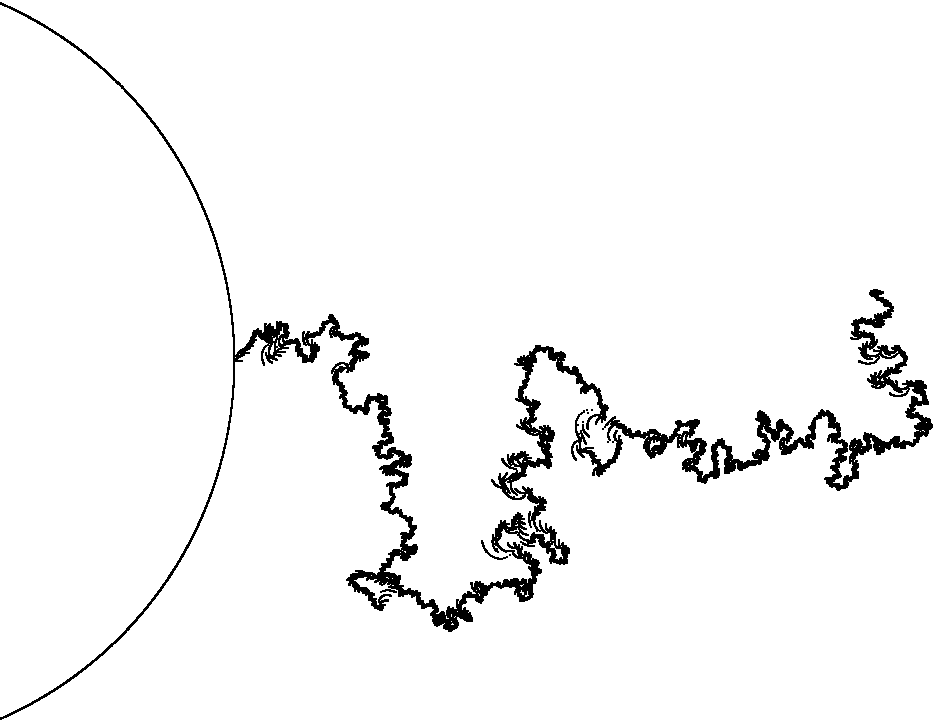}
 \caption{
\label{fig:sle-cluster}
 One cluster of the ALE$(0,-\infty)$ process,
 with 3000 particles each of capacity $\cparam = 0.0001$.
 } 
\end{figure}

In this paper,
we study the ALE model defined in
\autoref{sec:ale}
with
$\alpha = 0$
and large negative
values of the parameter
$\eta$,
which controls the influence of harmonic measure
on our attachment locations.



The case $\alpha = 2$ often gives a model in which
each particle is approximately the same size.
In this paper we take $\alpha = 0$,
where the model can be easier to analyse as the capacities
are deterministic.
In this case, the distortion of particles can lead to physically unrealistic
outcomes, as in \cite{george-hl}
where the distorted size of the final particle in the cluster
does not disappear in the limit.
For the model we are considering here,
\autoref{fig:last-slit} shows that the distortion
affects the shape as well as the size of the particles.\\

For $\eta > 0$ the density $h_{n+1}$ in \eqref{eq:density}
is an exaggeration of harmonic measure,
and in \cite{stv-ale}
the authors find that for $\eta > 1$
the attachment distribution is concentrated around
the point of highest harmonic measure,
converging to a single atom as $\cparam \to 0$.
For a slit particle the point of highest harmonic
measure is the tip (see \autoref{fig:harmonic-measure}),
so this corresponds to the growth of a straight line.

In this paper, we find the equivalent phase transition
in negative $\eta$:
for $\eta < -2$ the attachment distribution is concentrated
around the points of lowest harmonic measure.
For a slit particle the two points of lowest harmonic measure
are either side of the base (see \autoref{fig:harmonic-measure} again),
and so
$\theta_2 \approx \theta_1 \pm \beta$
with the probability of each tending to $1/2$
as $\cparam \to 0$.
We go on to find that for all $n$ the distribution of
$\theta_{n+1}$ is concentrated around $\theta_n \pm \beta$,
and so the angle sequence approximates a random walk
of step length $\beta \sim 2\cparam^{1/2}$.

This gives us the following statement about the \emph{driving function}
generating the cluster (see \autoref{sec:sle}):
\begin{proposition}
\label{thm:driver-limit}
	Fix some $T > 0$.
	For $\eta < -2$
	and if $\sigma(\cparam) \leq \sigmaub$
	for all $\cparam < 1$
	let $(\theta_n^\cparam)_{n \geq 1}$
	be the sequence of angles we obtain from the
	$\mathrm{ALE}(0,\eta)$ process
	with capacity parameter $\cparam$.
	Let $\deadtime = \inf\{ n \geq 2 : \min_{\pm}|\theta_n - (\theta_{n-1} \pm \beta_\cparam)| > D\}$,
	where $D = \cparam^{9/2}\sigma^{1/2}$.
	As $\cparam \to 0$,
	\[
		\P\left[ \deadtime \leq \rdown{T/\cparam} \right]
		= O(\cparam^3).
	\]
	Let $\xi_t^{\cparam} = \theta_{\rdown{t/\cparam}+1}^{\cparam}$
	for all $0 \leq t \leq T$.
	Then
	\[
		(\xi_t^{\cparam})_{t \in [0,T]} \to (2B_t)_{t \in [0,T]}
		\text{ in distribution as } \cparam \to 0,
	\]
	as a random variable in the Skorokhod space $D[0,T]$.
\end{proposition}

We explain in \autoref{sec:sle}
that by using Loewner's equation
we can immediately turn a result
about convergence of such a driving function
into a result about convergence of clusters
in an appropriate space $\mathcal{K}$.
The main theorem of the paper therefore follows
immediately from the proposition:

\begin{theorem}
\label{thm:main-result}
	For $\eta, \sigma$
	as in Proposition \ref{thm:driver-limit},
	let the corresponding $\mathrm{ALE}(0,\eta)$ cluster
	with $N = \rdown{T/\cparam}$
	particles each of capacity $\cparam$ be $K_N^\cparam$.
	Then as $\cparam \to 0$,
	$K_N^\cparam$ converges in distribution
	as a random variable in $\mathcal{K}$
	to a radial $\mathrm{SLE}_4$ curve of capacity $T$.
\end{theorem}

We can see in \autoref{fig:sle-cluster} a cluster corresponding to a random walk,
which despite being visibly composed of slits resembles an SLE$_4$ curve.



\begin{remark}
	We can give $\eta$ a physical interpretation if we think
	of growth in which access to environmental resources
	(proportional to harmonic measure)
	affects the growth rate in a \emph{non-linear} manner.
	For negative $\eta$ we could also interpret
	ALE$(\alpha,\eta)$ as modelling a cluster
	in an environment	which inhibits growth,
	so growth is concentrated
	in areas with the least exposure to the environment.
	
	The most physically-relevant models
	are those with $\alpha = 2$,
	where each particle in the cluster
	has approximately the same size.
	The case we consider, $\alpha = 0$,
	is somewhat unphysical as the later particles
	have a macroscopic size
	(in our case the final particle has
	a shape approximating the whole path of the SLE$_4$).
	In our case the ``visible'' portion of each particle
	which is not hidden between other particles
	is microscopic,
	although the ``visible'' part
	of the later particles
	is still significantly longer
	than the first particles.
	
 	In any case, the remarkable thing about the $\eta < -2$, $\alpha = 0$ case
 	is that it is drawn from a family of models which naturally extend
 	DLA-type growth,
 	and we obtain an SLE$_4$ scaling-limit for a whole range of parameters.
 	To this author's
 	knowledge no other
 	conformal growth model
 	in the plane has been rigorously proved
 	to converge to a random limit such as the SLE.
\end{remark}

\begin{remark}
	The convergence of attachment distributions
	to atomic measures for $\eta < -2$
	complements the phase transition result of \cite{stv-ale}
	in which it is shown that the limiting attachment measures
	are atomic for $\eta > 1$.
	For $-2 < \eta < 1$
	the distribution $h_2$ of the second particle
	is supported on all of $\T$ even in the limit
	$\cparam \to 0$,
	showing that we do indeed have three qualitative phases:
	for extreme values of $\eta$ the attachment measures
	are degenerate,
	but this is not the case for $-2 < \eta < 1$.
\end{remark}

In \autoref{sec:alt-particles} we conjecture that
similar scaling results to
Proposition \ref{thm:driver-limit}
and Theorem \ref{thm:main-result}
can be obtained with particles other than the slit.
Using suitable particles which have a single point of contact
with the circle, we believe that the limiting cluster
is always an SLE$_\kappa$ for some $\kappa \in [4,\infty]$
(where by SLE$_\infty$ we mean a uniformly growing disc).


\subsection{Loewner's equation and the Schramm--Loewner evolution}
\label{sec:sle}



We obtain a \emph{Schramm--Loewner evolution} (SLE) cluster
as the scaling limit of our model,
so we will give a brief overview here of what SLE is
and a few useful facts from Loewner theory which we use
to establish our scaling limit.
For a more detailed treatment, see
\cite{cm-aggregation}, \cite{lawler-conformal-book}
and \cite{duren-univalent}.\\

%

Firstly, we look at \emph{Loewner's equation},
which encodes our growing cluster
by a ``driving function'' taking values
on the circle.

%

\begin{definition}
Let $\xi : [0, T] \to \R$ be a \cadlag function.
Then there is a unique solution to \emph{Loewner's equation}
\begin{equation}
\label{eq:loewner}
	\varphi_0(z) = z,
	\quad
	\frac{\partial}{\partial t} \varphi_t(z) 
	=
	\varphi_t'(z) z \frac{z + e^{i\xi_t}}{z - e^{i\xi_t}},
	\quad
	z \in \Delta,
\end{equation}
corresponding to a growing cluster
via $\varphi_t(\Delta) = \C_\infty \setminus K_t$.\\
\end{definition}

For a sequence of angles $(\theta_n)_{n \geq 1}$
and a capacity $\cparam$,
a growth model constructed as in 
\autoref{sec:conformal-aggregation}
corresponds to the
cluster obtained
by solving Loewner's equation
with the driving function
\[
	\xi_t = \theta_{\rdown{t/\cparam}+1}.
\]

%

\begin{definition}
If $(B_t)_{t \in [0,T]}$ is a standard Brownian motion,
then the \emph{Schramm--Loewner evolution} with parameter $\kappa > 0$
(SLE$_\kappa$)
is the random cluster obtained by solving Loewner's equation
with the driving function given by $\xi_t = \sqrt{\kappa} B_t$.
\end{definition}

\begin{remark}
One very useful property of Loewner's equation for this paper
is that the map
$D[0,T] \to \mathcal{K}$ given by
$\xi \mapsto K_T$
is continuous \cite{levy-loewner-hulls},
where $D[0,T]$ is the usual Skorokhod space
and $\mathcal{K}$ is the set of compact subsets of $\C$ containing $0$,
equipped with the Carath\'{e}odory topology described in \cite{duren-univalent}.

This property of Loewner's equation means we can deduce
convergence of an ALE(0,$\eta$) cluster to an SLE$_4$
for $\eta < -2$
by showing that the cluster corresponds to a driving function
converging to $2B$ for a standard Brownian motion $B$
as $\cparam \to 0$.
\end{remark}

Schramm--Loewner evolutions describe the scaling limits
of many discrete models,
such as the loop-erased random walk,
which converges to an SLE$_2$ curve
\cite{lawler2011conformal},
or critical percolation,
the boundaries
of which has been related to SLE$_6$ \cite{smirnov-sle6}.
SLEs have also been used to construct the
\emph{quantum Loewner evolution} (QLE) \cite{qle} family of clusters,
which have been proposed as the scaling limits
of the dielectric breakdown model
on a number of random surfaces.

\subsection{Structure of paper}

Our proof of Proposition \ref{thm:driver-limit}
will involve showing that the distribution of $\theta_{n+1}$
conditional on the previous angles $(\theta_1, \dotsc, \theta_n)$
converges to $\frac{1}{2}( \delta_{\theta_n + \beta} + \delta_{\theta_n - \beta} )$,
and so the whole path $\xi^\cparam$ converges to the same limit as a simple random
walk with step length $\beta \sim 2\cparam^{1/2}$.\\ 

We can use a heuristic approach to see why we might expect this to be the case.
If we formally take $\eta = -\infty$ and $\sigma = 0$,
so the $n$th attachment point
$\theta_{n+1}$ is chosen uniformly from the finite set
$\{ \theta : \liminf_{\sigma \to 0} \inf_{w \in \T} | \Phi_n'(e^{\sigma}e^{i\theta}) |/| \Phi_n'(e^{\sigma}w) | > 0 \}$
(i.e. among the ``strongest poles'' of $\Phi_n'$),
and let $\tau = \inf\{ n : |\theta_{n} - \theta_{n-1}| \not= \beta \}$,
then we can calculate that for $N = \rdown{T/\cparam}$
in the limit $\cparam \to 0$ we have
$\P[ \tau \leq N ] \to 0$ as $\cparam \to 0$.
In other words, at each step the distribution $h_{n+1}$
is equal to $\frac{1}{2}( \delta_{\theta_n - \beta} + \delta_{\theta_n + \beta} )$.

Our approach for finite $\eta < -2$ will therefore be
to find a small upper bound on $h_{n+1}(\theta)$ for $\theta$
away from the poles of $\Phi_n'$
to deduce that $h_{n+1}$ is an approximation to a sum of atoms
at the poles.
Then we show separately
that the contribution to $Z_n = \int_\T h_{n+1}(\theta) \,\d\theta$
from poles other than $e^{i(\theta_n \pm \beta)}$ is small.

\begin{figure}
 \centering
 \includegraphics[width=0.46\textwidth, trim=0 0 0 0, clip]{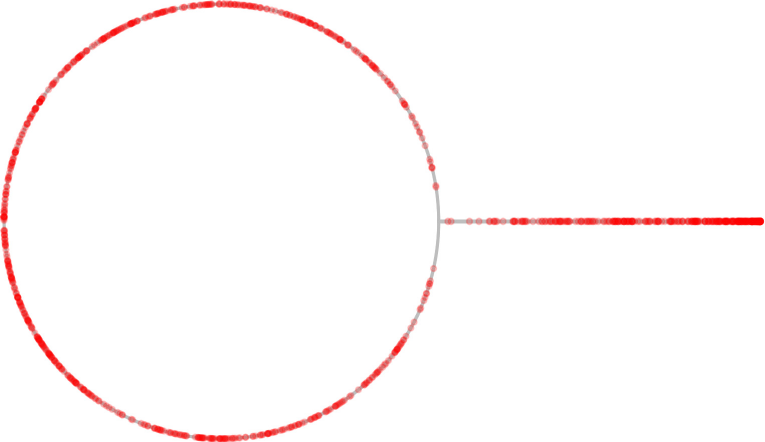}
 \hfill
 \includegraphics[width=0.46\textwidth, trim=0 0 0 0, clip]{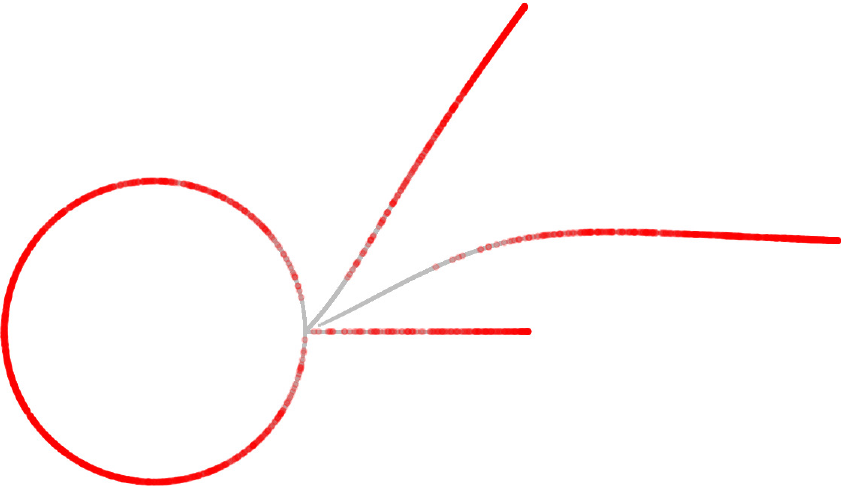}
 \caption{
 \label{fig:harmonic-measure}
 Left: the one-slit cluster of our process with 1,000 points in red sampled according to harmonic measure on the boundary.
 Right: the three-slit cluster of the process with 3,000 points sampled according to harmonic measure.
 Note in the second image that there are almost no points landing near the base
 of the most recent (longest) particle.}
\end{figure}

In the actual model with $-\infty < \eta < -2$, we can only show that
$h_{n+1}$ \emph{approximates}
$\frac{1}{2}(\delta_{\theta_n - \beta} + \delta_{\theta_n + \beta})$
as $\cparam \to 0$.
However, weak convergence of these measures
is not enough to prove Proposition \ref{thm:driver-limit},
so we will need to introduce some extra notation to describe the possible behaviour
of the process $(\theta_n)_{n \geq 1}$,
and make precise the way in which its steps converge to the SSRW steps as above.

\begin{definition}
\label{def:deadtime}
 For a small $D = D(\cparam)$
 (which we will specify later),
 define the stopping time
 \[
  \deadtime
  :=
  \inf\{
   n \geq 2
   :
   \min(
    | \theta_n - (\theta_{n-1} + \beta) |,
    | \theta_n - (\theta_{n-1} - \beta) |
   )
   > D
  \}.
 \]
\end{definition}

\begin{remark}
 Given that $n < \deadtime$, we have a lot of information about the angle
 sequence $(\theta_1, \dotsc, \theta_n)$,
 and so can say quite a lot about the conditional distribution of $\theta_{n+1}$.
 In particular, we can say that the probability that $n+1 = \deadtime$ is very low,
 and that the distribution of $\theta_{n+1} - \theta_n$ is (approximately) symmetric.
 The results of all the following sections will be used to establish these two facts.
\end{remark}

\begin{theorem}
\label{thm:step-convergence}
 Suppose that $\nu > 2$.
 There exists a constant $A > 0$
 depending only on $\nu$ and $T$
 such that when
 $\sigma \leq
 \cparam^{2^{2^{1/\cparam}}}$,
 then for $D = \cparam^{9/2}\sigma^{1/2}$,
 whenever $n < N \wedge \deadtime$
 and $\cparam$ is sufficiently small,
 \begin{equation}
 \label{eq:concentration-result}
  \int_{F_n}
   h_{n+1}(\theta)
  \,\d \theta
  \leq
  A \cparam^{4}
 \end{equation}
 with probability $1$,
 where
 $F_n = \{
  \theta \in \T
  :
  |\theta - (\theta_n + \beta)| \geq D
  \text{ and }
  |\theta - (\theta_n - \beta)| \geq D
 \}$,
 and with probability $1$
 \begin{equation}
 \label{eq:symmetry-result}
  \left|
   \int_{\theta_n + \beta - D}^{\theta_n + \beta + D}
    h_{n+1}(\theta)
   \,\d\theta
   -
   \int_{\theta_n - \beta - D}^{\theta_n - \beta + D}
    h_{n+1}(\theta)
   \,\d\theta
  \right|
  \leq
  A \cparam^{11/4}.
 \end{equation}
\end{theorem}

In \autoref{sec:tracking} we prove a number of technical
results about the positions of the images and preimages of points
$w \in \Delta$ under the maps $f_j$, $\Phi_{n} = f_1 \circ \dotsb \circ f_n$,
and $\Phi_{j,n} = \Phi_j^{-1} \circ \Phi_n$
when $w$ is close to the poles of $\Phi_n'$.
When dealing with points away from these poles,
we make extensive use of results from \cite{stv-ale}.
Our estimates for the positions of these images
will be useful when we find upper bounds on the derivative
$|\Phi_n'(w)| = |f_n'(w)| \times |f_{n-1}'(\Phi_{n-1,n}(w))| \times \dotsm \times |f_1'(\Phi_{1,n}(w))|$,
using lower bounds on the distance between $\Phi_{j,n}(w)$
and the poles of $f_j'$.

In \autoref{sec:partition-function} we integrate
the pre-normalised density $| \Phi_n'(e^{\sigma + i\theta}) |^{\nu}$
over the regions around $\theta_n \pm \beta$,
and so obtain a lower bound on
\[
Z_n = \int_\T | \Phi_{n}'(e^{\sigma + i \theta})|^{\nu} \,\d \theta.
\]

In \autoref{sec:concentration} and \autoref{sec:small-contributions}
we find upper bounds on $|\Phi_n'(e^{\sigma + i\theta})|$
for $\theta \in F_n$,
and so using the lower bound on $Z_n$ we can establish the bound
\eqref{eq:concentration-result}.

In \autoref{sec:symmetry} we establish the technical results
needed to prove \eqref{eq:symmetry-result}.

\begin{remark}
 In our proof of Theorem \ref{thm:step-convergence},
 the convergence of $h_{n+1}$ to
 $\frac{1}{2}( \delta_{\theta_n + \beta} + \delta_{\theta_n - \beta})$
 does not rely on the convergence of $h_1, \dotsc, h_n$ to these
 \emph{symmetric} discrete measures, only that $n < \deadtime$.
 If we were to use the fact that the angle sequence up until time $n$
 is very close to a simple symmetric random walk,
 then some properties (such as the fact that the longest interval
 on which a SSRW is monotone has length of order $O(\log n)$)
 would allow us to optimise our choice of $\sigma$ further than we have.
 However, for the convergence of our cluster to an SLE$_4$ curve,
 we do require a $\sigma$ which decays at least as quickly as $\cparam^{1/\cparam}$,
 which is already much faster than the fixed power of $\cparam$ used in
 \cite{stv-ale} and elsewhere,
 so we have not attempted to optimise our choice of
 $\sigma \leq \cparam^{2^{2^{1/\cparam}}}$.
 
 If $\sigma$ decays more slowly than $\cparam^{1/\cparam}$,
 but more quickly than $\cparam^{1/2}$,
 then
 heuristic arguments suggest that
 there is a period in which the driving function is a random walk,
 and then a period where the growth is measurable with respect to the random walk
 (i.e. a period of random growth and then a period of deterministic growth).
 We do not believe the resulting cluster
 converges to any
 known object as $\cparam \to 0$.
\end{remark}

In \autoref{sec:alt-particles}
we define a family of particles for which
we believe
analogous versions of our main scaling result
Theorem \ref{thm:main-result}
holds.
We conjecture that
suitably constructed ALE($0,\eta$) models
with $\eta < -2$
will converge to either an SLE$_\kappa$
with $\kappa \geq 4$, or to a uniformly growing disc.
We also believe that every $\kappa \geq 4$ is attained
by this family.

\subsection{Table of notation}
\label{sec:notation}

As we introduce a lot of notation in this paper,
we will give a list here
so that it is possible to look up any
notation appearing in any section
without searching for where it was introduced.\\

\textbf{Subsets of the complex plane}
\begin{description}[leftmargin=1cm, style=nextline]
\item[$\C_\infty$]
The Riemann sphere, $\C \cup \{ \infty \}$
\item[$\D$]
The open unit disc
$\{ z \in \C : |z| < 1 \}$.
\item[$\cudisc$]
The closed unit disc
$\{ z \in \C : |z| \leq 1 \}$.
\item[$\Delta$]
The exterior disc
$\C_\infty \setminus \cudisc$.
\item[$\T$]
The unit circle
$\partial \Delta
=
\{ z \in \C : |z| = 1 \}
=
\{ e^{i\theta} : \theta \in \R \}$.
We will often abuse notation and identify $\T$ with $\R / 2\pi\Z$.
\item[$\partial U$]
The boundary of a set
$U \subseteq \C_\infty$,
defined as
$\partial U
=
\overline{U} \setminus U^\circ$.

\end{description}


\textbf{Conformal maps}
\begin{description}[leftmargin=1cm, style=nextline]
\item[$f$]
The conformal map $f_\cparam \colon \Delta \to \Delta \setminus (1,1+d(\cparam)]$ which we say attaches a particle to the unit circle at the point 1.
\item[$f_j$]
Given a sequence of angles $(\theta_j)_{j \geq 1}$,
$f_j$ attaches a particle
to the unit circle at the point
$e^{i\theta_j}$, so $f_j(z) := e^{i\theta_j}f(e^{-i\theta_j}z)$.
\item[$\beta$]
The distance from $1$
of the points which
are sent to the base
of the particle by $f$.
Defined uniquely as the
$\beta = \beta(\cparam)
\in (0,\pi)$ such that
$f_\cparam(e^{\pm i\beta}) = 1$, 
and obeys
$\beta \sim 2\cparam^{1/2}$
as $\cparam \to 0$.
\item[$d$]
The length of the particle attached by $f$,
defined by $f_\cparam(1) = 1+d(\cparam)$.
Obeys $d \sim \beta \sim 2\cparam^{1/2}$ as $\cparam \to 0$.
\item[$\Phi_n$]
The conformal map which attaches the entire cluster of $n$ particles to the unit circle at the point $1$. Constructed as $f_1 \circ f_2 \circ \dotsb \circ f_n$.
\item[$\Phi_{j,n}$]
The conformal map which attaches
only the most recent $n-j$
particles to the unit circle.
Given by
$\Phi_{j,n} = \Phi_j^{-1} \circ \Phi_n$.
\end{description}

\textbf{Model parameters}
\begin{description}[leftmargin=1cm, style=nextline]
\item[$\eta$]
The parameter controlling the relationship between
our attachment distributions and the harmonic measure
on the boundary of the cluster.
Throughout this paper we take $\eta < -2$.
\item[$\nu$]
We write $\nu = -\eta$. Note that $\nu > 2$ throughout.
\item[$T$]
The total capacity of our cluster, fixed throughout.
\item[$\cparam$]
The capacity of each individual particle attached to the cluster.
We consider in this paper the limit $\cparam \to 0$, so all the following parameters are functions of $\cparam$.
\item[$\sigma$]
A regularisation parameter,
used so that we do not evaluate our conformal maps $\Phi_n'$
at their poles on $\T$, instead evaluating everything on $e^\sigma \T$.
We take $\sigma$ to be a function of $\cparam$,
decaying very rapidly as $\cparam \to 0$:
$\sigma \leq \cparam^{2^{2^{1/\cparam}}}$
\item[$L$]
The maximum distance of $z$ from
$e^{i(\theta_n \pm \beta)}$
at which we rely on
the estimates for $|\Phi_{j,n}(z) - e^{i\theta_{j+1}}|$ we obtain
in the proof of Theorem \ref{thm:pf-bound}.
We take $L$ to be a function of $\cparam$
which does not decay as rapidly as $\sigma$:
$L = \cparam^{2^{N+1}}$.
\item[$D$]
A bound on $\min_{\pm}| \theta_{n+1} - (\theta_n \pm \beta) |$
which holds with high probability.
If this distance exceeds $D$, we stop the process.
We can take $D = \cparam^{9/2} \sigma^{1/2}$.
\end{description}

\textbf{Points in $\T$}
\begin{description}[leftmargin=1cm, style=nextline]
\item[$\thetalegit_j$]
The point in $\T$
which $\theta_j$ was ``supposed to'' attach nearby to,
i.e.\ the unique choice of $\theta_{j-1} \pm \beta$ which
$\theta_j$ is within $D$ of
(if $\theta_j$ is not within $D$ of either,
we will have stopped the process at time $\deadtime \leq j$).
\item[$\thetafake_j$]
The choice of $\theta_{j-1}\pm \beta$ which isn't $\theta_j^\top$.
\item[$\bp^n_j$]
The point on $\T$ corresponding to
the base of the
$j$th 
particle in the cluster $K_n$,
for
$1 \leq j \leq n-1$.
Given by $\bp^n_j
:=
\Phi_{j,n}^{-1}(
e^{i\thetafake_{j+1}}
)$.
See \autoref{fig:composition} for an illustration.
We refer to the points on $\T$ close to $\bp_j^n$ for some $j$
as \emph{singular points} for $h_{n+1}$,
and points away from all $\bp_j^n$ as \emph{regular points}.
\end{description}

\textbf{Probabilistic objects}
\begin{description}[leftmargin=1cm, style=nextline]
\item[$h_{n+1}$]
The density of the distribution on $\T$ of $\theta_{n+1}$, conditional on $\theta_1, \dotsc, \theta_{n}$.
Given by $h_{n+1}(\theta)
\propto
|\Phi_{n}'(e^{\sigma + i\theta})|^{\nu}$.
\item[$Z_n$]
The normalising factor for $h_{n+1}$.
Given by
$Z_n
:=
\int_\T
| \Phi_n'(e^{\sigma + i \theta}) |^{\nu}
\,\d\theta$.
\item[$\P$]
The law of $(\theta_n)_{n \in \N}$.
Implicitly depends on $\cparam$ and $\sigma$.
\item[$\deadtime$]
The first time at which some $\theta_{n+1}$
is further than $D$ from
both of $\theta_n \pm \beta$.
We stop the process when this happens,
but show in \autoref{sec:new-basepoints}
and \autoref{sec:small-contributions}
that with high probability
$\deadtime > N := \rdown{T/\cparam}$.
\end{description}

\textbf{Approximations and bounds}\\
We will use the following notation
when we have two functions depending on a parameter $x$
which is converging to some $x_0 \in \R\cup\{ \pm \infty\}$,
and we want to say the two functions are similar in some way,
or that one bounds the other.
\begin{description}[leftmargin=3cm, style=nextline]
\item[$f(x) \sim g(x)$]
The ratio $\frac{f(x)}{g(x)} \to 1$ as $x \to x_0$.
\item[$f(x) = O(g(x))$]
The ratio
$\left|\frac{f(x)}{g(x)}\right|$ is bounded above as $x \to x_0$,
so there exists a constant $C > 0$ such that
$|f(x)| \leq C|g(x)|$ in a neighbourhood of $x_0$.
The constant $C$ should not depend on any other parameter
or variable.
If the value of $C$ does depend on a parameter $\rho$,
we will write $f(x) = O_\rho( g(x) )$.
Throughout this paper we hold $T$ and $\nu = -\eta$ fixed,
so we may occasionally omit these as subscripts when the constant depends on them.
\item[$f(x) = o(g(x))$]
The ratio $\left| \frac{f(x)}{g(x)}\right| \to 0$ as $x \to x_0$.
\end{description}

When $f$ and $g$ are non-negative
(particularly when they are probabilities
or densities),
we may use the following alternative notations.
\begin{description}[leftmargin=3cm, style=nextline]
\item[$f(x) \lesssim g(x)$]
The same as $f(x) = O(g(x))$,
i.e.\ there exists a constant $C>0$
such that
$f(x) \leq C g(x)$
in a neighbourhood of $x_0$.
\item[$f(x) \ll g(x)$]
The same as $f(x) = o(g(x))$,
i.e.\ $f(x)/g(x) \to 0$
as $x \to x_0$.
\item[$f(x) \asymp g(x)$]
Both $f(x) = O(g(x))$ and $g(x) = O(f(x))$,
i.e.\ there exists constants
$C_1, C_2 > 0$ such that\\
$C_1 g(x) \leq f(x) \leq C_2 g(x)$
in a neighbourhood of $x_0$.
\end{description}
Finally, we may write $f(x) \approx g(x)$,
but this will only be used informally
to mean that $f$ and $g$ behave similarly
in some sense.

\section{Spatial distortion of points}
\label{sec:tracking}

There are several steps we need
to establish our upper bound on $\int h_{n+1}(\theta)\,\d\theta$ in
\eqref{eq:concentration-result},
including precise estimates for $|\Phi_n'|$ near its poles.
We can decompose the derivative
\begin{equation}
\label{eq:cr-decomp}
 \Phi_n'(w)
 =
 \prod_{j=0}^{n-1}
  f_{n-j}'(
   \Phi_{n-j,n}(w)
  )
\end{equation}
where
\begin{equation}
\label{eq:Phikn}
 \Phi_{k,n}
 :=
 \Phi_{k}^{-1} \circ \Phi_n
 =
 f_{k+1} \circ f_{k+2} \circ \dotsc \circ f_n.
\end{equation}
Then we have precise estimates on
$|f'|$ near to its poles $e^{\pm i \beta}$,
and upper bounds away from these poles,
and so we write
\begin{equation}
\label{eq:chain-rule-expansion}
 |\Phi_n'(w)|
 =
 \prod_{j=0}^{n-1}
  \left|f'\left(
   e^{-i\theta_{n-j}}
   \Phi_{n-j,n}(w)
   \right)
  \right|.
\end{equation}
We will show that if $w$ is close
to one of $e^{i(\theta_n \pm \beta)}$,
then for each $j$,
the point
$e^{-i\theta_{n-j}}\Phi_{n-j,n}(w)$
is close to a pole of $|f'|$,
and we will derive specific estimates
on the distance in terms of the distance
$|w - e^{i(\theta_n \pm \beta)}|$.
Conversely, we will show that the only way
for \emph{every} image
$e^{-i\theta_{n-j}}\Phi_{n-j,n}(w)$
to be close to a pole
is for $w$ to be
close to $e^{i(\theta_n \pm \beta)}$,
and so the measure $\d h_{n+1}$ is
concentrated around
$\theta_n + \beta$ and $\theta_n - \beta$.

Firstly, we will establish an estimate
for $|f'|$ close to its poles $e^{\pm i \beta}$,
and a universal upper bound away from
these two points.

\begin{lemma}
\label{thm:f-prime-estimate} 
 There are universal constants $A_1, A_2 > 0$ such that
 for all $\cparam < 1$,
 for $w \in \Delta$,
 if $|w - e^{ i\beta}| \leq \frac{3}{4}\beta$,
 then
 \begin{equation}
 \label{eq:f-prime-estimate}
  A_1 \frac{\beta^{1/2}}{|w - e^{i\beta}|^{1/2}}
  \leq
  |f_\cparam'(w)|
  \leq
  A_2 \frac{\beta^{1/2}}{|w - e^{i\beta}|^{1/2}},
 \end{equation}
 and similarly if
 $|w - e^{-i\beta}| \leq \frac{3}{4}\beta$.
 
 Moreover, there is a third constant $A_3$
 such that if
 $\min \{ | w - e^{i\beta} |, | w - e^{-i\beta} | \}
 > \frac{3}{4}\beta$,
 then
 \[
  |f_\cparam'(w)| \leq A_3.
 \]
\end{lemma}

\begin{proof}
 See Lemma 5 of \cite{stv-ale}.
\end{proof}

This lemma tells us that
the derivative $| \Phi_n'(w) |$
will be large only when many of the points
$e^{-i\theta_{n-j}}\Phi_{n-j,n}(w)$
in \eqref{eq:chain-rule-expansion} are
close to one of the poles $e^{\pm i \beta}$.
We will next introduce some technical estimates
which will allow us to determine
for which points $w$ this is true.

\begin{remark}
If we imagine an idealised path in which
$|\theta_{i+1} - \theta_i| = \beta$
for all $i$, then
$f_n(e^{i(\theta_n \pm \beta)})
= e^{i\theta_{n-1}}$,
and 
$f_{n-1}(e^{i\theta_{n-1}})
= e^{i\theta_{n-2}}$,
and so on.
Hence
$\Phi_{n-j,n}(e^{i(\theta_n \pm \beta)})
=
e^{i\theta_{n-j+1}}
=
e^{i(\theta_{n-j} + s_{n-j} \beta)}$,
where $s_{n-j} \in \{ \pm 1 \}$.
So if a point $w$ is close to one of
$e^{i(\theta_n \pm \beta)}$ 
then, as $f$ is continuous when
extended to $\overline{\Delta}$,
each of the points in
\eqref{eq:chain-rule-expansion}
is close to
$e^{i s_{n-j} \beta}$,
but continuity alone does not allow us to make precise
what we mean by ``$w$ is close to $e^{i(\theta_n \pm \beta)}$'',
so to estimate the size of $|\Phi_n'(w)|$,
we need a precise estimate for
$|f(w) - f(e^{i\beta})|$
in terms of
$|w - e^{i\beta}|$.
\end{remark}

\begin{lemma}
\label{thm:distance-estimate}
 For $w \in \Delta$,
 for all $\cparam < 1$,
 if $| w - e^{i\beta} | \leq \beta / 2$,
 then
 \begin{equation}
 \label{eq:distance-estimate}
  \begin{aligned}
   |f_\cparam(w) - 1|
   =\ 
   &2(e^{\cparam} - 1)^{1/4}
   |w - e^{i\beta}|^{1/2}\\
   &\times
   \left(
    1 +
    O\left[
     \frac{|w - e^{i\beta}|}{\cparam^{1/2}}
     \vee
     \cparam^{1/4} |w - e^{i\beta}|^{1/2}
    \right]
   \right).
  \end{aligned}
 \end{equation}
\end{lemma}

\begin{proof}
 We will work with the half-plane
 slit map
 $\t{f}_\cparam
 \colon \H
 \to
 \H
 \setminus
 (0, i \sqrt{1 - e^{-\cparam}}\,]$
 by conjugating $f$ with the
 M\"obius map
 $m_\H \colon \Delta \to \H$
 given by
 \begin{equation}
 \label{eq:ed-to-hp}
  m_{\H}(w) = i \frac{w - 1}{w + 1},
 \end{equation}
 and its inverse
 \begin{equation}
 \label{eq:hp-to-ed}
  m_{\Delta}(z)
  := m_{\H}^{-1}(z)
  = \frac{1 - iz}{1 + iz}.
 \end{equation}
 The benefit of this is that
 $\t{f}_\cparam$ has a simple
 explicit form:
 \begin{equation}
 \label{eq:half-plane-slit-map}
  \t{f}_\cparam(\zeta)
  =
  e^{-\cparam/2}
  \sqrt{
   \zeta^2 - (e^{\cparam} - 1)
  }
 \end{equation}
 where the branch of the
 square root is given by
 $\arg \colon
 \C \setminus [0, \infty)
 \to
 (0, 2\pi)$,
 so we write
 \[
  f_\cparam = m_\Delta \circ \t{f}_\cparam \circ m_\H
 \]
 and will derive a separate estimate
 for each of the three maps.
 
 As $w$ is close to
 $e^{i\beta} = 2e^{-\cparam} - 1
 + 2i e^{-\cparam}\sqrt{e^\cparam - 1}$,
 we will expand each map about
 the images (given by a simple calculation)
 $m_\H(e^{i\beta}) = -\sqrt{e^\cparam - 1}$,
 $\t{f}_\cparam( - \sqrt{e^\cparam - 1} ) = 0$,
 and
 $m_\Delta(0) = 1$.
 Our calculations will show that $m_\Delta$
 and $m_\H$ behave like scaling by a constant
 close to the relevant points, and that the
 behaviour of $f_\cparam$ seen in
 \eqref{eq:distance-estimate}
 is due to the behaviour of $\t{f}_\cparam$
 close to $\pm \sqrt{e^\cparam - 1}$.
 
 First, when $w = e^{i\beta} + \delta$,
 \begin{align}
  \nonumber
  | m_\H(w) - m_\H(e^{i\beta}) |
  &=
  \left|
   \frac{2\delta}{
    (e^{i\beta}+1+\delta)(e^{i\beta} + 1)
   }
  \right|\\
  \label{eq:hp-estimate}
  &=
  \frac{1}{2}e^\cparam
  |\delta|
  (1 + O(|\delta|))
 \end{align}
 since a simple calculation shows that
 $|e^{i\beta} + 1|^2 = 4e^{-\cparam}$.
 
 Next, we will evaluate $\t{f}_\cparam$
 at a point close to one of the two preimages
 of $0$, $\pm \sqrt{e^\cparam - 1}$:
 \begin{align}
  \nonumber
  \left|
   \t{f}_\cparam( \pm\sqrt{e^\cparam - 1} + \lambda )
  \right|
  &=
  e^{-\cparam/2}
  \left|
   \sqrt{
    \pm 2\sqrt{e^\cparam - 1} \lambda + \lambda^2
   }
  \right|\\
  \label{eq:hpslit-estimate}
  &=
  \sqrt{2} e^{-\cparam/2}
  (e^\cparam - 1)^{1/4}
  |\lambda|^{1/2}
  \left(
   1 +
   O\left(
    \frac{|\lambda|}{
    \cparam^{1/2}
    }
   \right)
  \right).
 \end{align}
 
 Finally, for a small $z \in \H$,
 \begin{align}
 \label{eq:ed-estimate}
  |m_\Delta(z) - 1|
  =
  \left|
   \frac{1 - iz}{1 + iz} - 1
  \right|
  =
  \left|
   \frac{-2iz}{1+iz}
  \right|
  =
  2|z|( 1 + O(|z|) ).
 \end{align}
 
 Then for $w$ close to $e^{i\beta}$,
 applying \eqref{eq:hp-estimate},
 \eqref{eq:hpslit-estimate} and
 \eqref{eq:ed-estimate} in turn,
 we obtain
 \begin{align*}
  |f(w) - 1|
  =\
  &2(e^\cparam - 1)^{1/4}
  |w - e^{i\beta}|^{1/2}\\
  \times
  &\left(1 + O\left(
   \frac{|w - e^{i\beta}|}{\cparam^{1/2}}
  \right)\right)
  \left(1 + O\left(
   \cparam^{1/4}|w - e^{i\beta}|^{1/2}
  \right)\right).
 \end{align*}
 Then for $\cparam^{3/2} \leq |w - e^{i\beta}| \leq \beta/2$,
 we have the estimate \eqref{eq:distance-estimate}
 with error term of order $\cparam^{-1/2} |w - e^{i\beta}|$,
 and for $|w - e^{i\beta}| \leq \cparam^{3/2}$
 the error term has order $\cparam^{1/4} |w-e^{i\beta}|^{1/2}$.
\end{proof}

\begin{remark}
 Unlike most results in this section,
 we will not use the following lemma
 in Section \ref{sec:new-basepoints},
 but it will be very useful in Section \ref{sec:old-basepoints}.
 We include it here and omit the proof as it is very similar 
 to Lemma \ref{thm:distance-estimate}.
\end{remark}

\begin{lemma}
\label{thm:inverse-distance-estimate}
 For all $\cparam < 1$,
 if $z \in \Delta \setminus (1, 1+d(\cparam)]$
 has $|z-1| \leq \cparam$, then
 \[
  \min_{\pm} | f^{-1}(z) - e^{\pm i \beta} |
  =
  \frac{ | z - 1 |^2 }{ 4(e^\cparam - 1)^{1/2} }
  \left(
   1 + O\left(
    |z - 1|
   \right)
  \right).
 \]
\end{lemma}

Now we have all the technical results we need
in order to prove our lower bound on
$| \Phi_n'(w) |$
when $w$ is close to one of the two
``most recent basepoints''
$e^{i(\theta_n \pm \beta)}$.
We will derive the bound itself
in \autoref{sec:partition-function},
and here we will show that
each of the points
$\Phi_{n-j,n}(w)$
in \eqref{eq:chain-rule-expansion}
is close to $e^{i\theta_{n-j+1}}$.

\begin{proposition}
\label{thm:sticky}
 Let $L = L(\cparam, N) = \cparam^{2^{N+1}}$,
 and let $n < N \wedge \deadtime$.
 If $\delta :=
 \min | w - e^{i(\theta_n \pm \beta)} | \leq 2L$,
 and $|w| \geq e^{\sigma}$,
 then for all $1 \leq j \leq n$,
 \begin{equation}
 \label{eq:sticky}
  \left|
   \Phi_{n-j,n}(w) - e^{i\thetalegit_{n-j+1}}
  \right|
  =
  \left[
   \discoeff
  \right]^{2(1 - 2^{-j})}
  \delta^{2^{-j}}
  (1 + O(\cparam^4)).
 \end{equation}
\end{proposition}

Before we begin the proof we will introduce some notation
in order to make the argument easier to follow.

\begin{definition}
\label{def:thetalegit}
\label{def:thetafake}
 By definition of $\deadtime$,
 for each $n < \deadtime$
 one of the two angles $\theta_{n-1} \pm \beta$
 is within distance $D$ of $\theta_n$.
 We will call the closer of the two angles $\thetalegit_n$,
 and the other angle $\thetafake_n$.
\end{definition}

\begin{proof}[Proof of Proposition \ref{thm:sticky}]
 We will proceed by induction on $j$.
 For $j = 1$, the estimate \eqref{eq:sticky}
 follows directly from Lemma \ref{thm:distance-estimate}.
 For a given $1 \leq j \leq n-1$,
 assume that
 \[
  | \Phi_{n-j,n}(w) - e^{i\thetalegit_{n-j+1}} |
  =
  \left[
   \discoeff
  \right]^{2(1 - 2^{-j})}
  \delta^{2^{-j}}
  (1 + O(\cparam^4)),
 \]
 (as $| \theta_n - \thetalegit_n | < D \ll \cparam^4$,
 this certainly holds for $j = 1$)
 and then by the triangle inequality,
 since $| e^{i\theta_{n-j}} - e^{i\thetalegit_{n-j}}|
 \leq
 | \theta_{n-j} - \thetalegit_{n-j} |
 < D$, we have
 \begin{align*}
  | \Phi_{n-j-1,n}(w) - e^{i\theta_{n-j}} |
  -
  D
  \leq
  | &\Phi_{n-j-1,n}(w) - e^{i\thetalegit_{n-j}} | \leq
  | \Phi_{n-j-1,n}(w) - e^{i\theta_{n-j}} |
  +
  D.
 \end{align*}
 Now by Lemma \ref{thm:distance-estimate},
 \begin{align*}
  | &\Phi_{n-j-1,n}(w) - e^{i\theta_{n-j}} |
  =
  | f_{n-j}(\Phi_{n-j,n}(w)) - f_{n-j}(e^{i\thetalegit_{n-j+1}}) |\\
  &=
  | f( e^{-i\thetalegit_{n-j+1}} \Phi_{n-j,n}(w) ) - 1 |\\
  &=
  \discoeff
  | e^{-i\thetalegit_{n-j+1}} \Phi_{n-j,n}(w) - 1 |^{1/2}
  ( 1 +
   O(
    \cparam^{1/4} | e^{-i\thetalegit_{n-j+1}} \Phi_{n-j,n}(w) - 1 |^{1/2}
   )
  )\\
  &=
  \left[ \discoeff \right]^{
   1 + (1 - 2^{-j})
  }
  \delta^{2^{-(j+1)}}
  (1 +
   O(
    \cparam^4
   )
  )
  (1 +
   O(
    \cparam^{3/8} \delta^{2^{-(j+1)}}
   )
  )\\
  &=
  \left[ \discoeff \right]^{
   2(1 - 2^{-(j+1)})
  }
  \delta^{2^{-(j+1)}}
  (1 +
   O(
    \cparam^4
   )
  )
 \end{align*}
 and the second error term is absorbed since
 $\delta^{2^{-(j+1)}}
 \leq
 (2L)^{2^{-(j+1)}}
 \leq
 \cparam^4$.
 
 Now as
 $\delta
 =
 |w - e^{i(\theta_n \pm \beta)}|
 \geq
 |w| - 1
 \geq \sigma$,
 and $D \sim \cparam^{9/2}\sigma^{1/2}$
 (see \autoref{sec:notation}),
 we have
 \begin{align*}
  | \Phi_{n-j-1,n}(w) - e^{i\thetalegit_{n-j}} |
  &=
  | \Phi_{n-j-1,n}(w) - e^{i\theta_{n-j}} |
  \left(1 + 
   O\left(
    \frac{D}{\cparam^{\frac{1}{2}(1 - 2^{-(j+1)})} \delta^{2^{-(j+1)}}}
   \right)
  \right)\\
  &=
  | \Phi_{n-j-1,n}(w) - e^{i\theta_{n-j}} |
  \left(1 + 
   O\left(
    \cparam^4 \sigma^{1/4}
   \right)
  \right),
 \end{align*}
 and hence our result holds for all $1 \leq j \leq n$
 by induction.
\end{proof}

\section{The newest basepoints}
\label{sec:new-basepoints}

\subsection{A lower bound on the normalising factor}
\label{sec:partition-function}

We defined in \eqref{eq:density} the density function $h_{n+1}(\theta)$ and the $n$th normalising factor
\begin{equation}
\label{eq:partition-function}
Z_n = \int_{\T} | \Phi_{n}'(e^{\sigma + i\theta})|^{-\eta}\,\d \theta.
\end{equation}
If we are going to find upper bounds on $h_{n+1}$ by bounding $|\Phi_n'|$,
then we will need to have some lower bound on the normalising factor $Z_n$.
In this section, we will obtain a lower bound on $Z_n$,
and it will give us our upper bound on $h_{n+1}$ in \autoref{sec:old-basepoints}.
First, we will need a good estimate for $|\Phi_n'|$
around the main poles $e^{i(\theta_n \pm \beta)}$.

\begin{lemma}
\label{thm:deriv-estimate}
 Let $n < \rdown{T/\cparam} \wedge \deadtime$.
 There are constants $A_1, A_2 > 0$ such that
 for any $\cparam < 1$,
 whenever $|\varphi| < L$,
 \[
  A_1^n \frac{\cparam^{\frac{1}{2}(1 - 2^{-n})}}{
  ( \sigma^2 + \varphi^2 )^{\frac{1}{2}(1 - 2^{-n})}
  }
  \leq
  \left|
   \Phi_n'\left(
    e^{\sigma + i( \theta_n \pm \beta + \varphi )}
   \right)
  \right|
  \leq
  A_2^n \frac{\cparam^{\frac{1}{2}(1 - 2^{-n})}}{
  ( \sigma^2 + \varphi^2 )^{\frac{1}{2}(1 - 2^{-n})}
  }
 \]
 provided that $\sigma = \sigma(\cparam) \leq
 L$.
\end{lemma}

\begin{proof}
 For $|\varphi| < L$, without loss of generality take
 $\theta = \theta_n + \beta + \varphi$.
 Since $\Phi_n = f_1 \circ \dotsb \circ f_n$, by the chain rule,
 \[
  |\Phi_n'(e^{\sigma + i\theta})|
  =
  \prod_{j=0}^{n-1}
   \left|f'\left(
    e^{-i\theta_{n-j}}
    \Phi_{n-j,n}(e^{\sigma + i\theta})
    \right)
   \right|,
 \]
 where $\Phi_{k,n} = \Phi_k^{-1} \circ \Phi_n
 = f_{k+1} \circ f_{k+2} \circ \dotsb \circ f_n$.
 
 By Proposition \ref{thm:sticky},
 if $\delta
 :=
 |e^{\sigma+i\theta} - e^{i(\theta_n + \beta)}|
 <
 2L$,
 then for all $1 \leq j \leq n-1$,
 $|\Phi_{n-j,n}(e^{\sigma + i\theta}) - e^{i\thetalegit_{n-j+1}}|
 =
 [ \discoeff ]^{2(1 - 2^{-j})}
 \delta^{2^{-j}}
 (1 + O(\cparam^4))$,
 and so 
 by Lemma \ref{thm:f-prime-estimate}
 (the above estimate shows that
 $e^{-i\theta_{n-j}}\Phi_{n-j,n}(e^{\sigma + i\theta})$
 is close enough to one of $e^{\pm i \beta}$
 to apply this lemma),
 \begin{align*}
  \left|
   f'\left(
    e^{-\theta_{n-j}}\Phi_{n-j,n}(e^{\sigma + i\theta})
   \right)
  \right|
  &\asymp
  \beta^{1/2}
  |\Phi_{n-j,n}(e^{\sigma + i \theta}) - e^{i\thetalegit_{n-j+1}}|^{-1/2}\\
  &=
  \beta^{1/2}
  [\discoeff]^{-(1-2^{-j})}
  \delta^{-2^{-j-1}}
  (1 + O(\cparam^4))\\
  &\asymp
  \cparam^{2^{-(j+2)}}
  \delta^{-2^{-(j+1)}}.
 \end{align*}
 
 For $j = 0$, as $\Phi_{n,n}$ is the identity map,
 \begin{align*}
  | f'(e^{-i\theta_{n}}\Phi_{n,n}(e^{\sigma+i\theta}))|
  =
  | f'( e^{\sigma+i(\theta - \theta_n)} ) |
  &\asymp
  A_1 \beta^{1/2} \delta^{-1/2}\\
  &\asymp
  A \cparam^{1/4} \delta^{-1/2}.
 \end{align*}
 
 Now if we combine the bounds for each term
 in the above product for $|\Phi_{n}'(e^{\sigma + i\theta})|$,
 we have
 \begin{align*}
  |\Phi_{n}'(e^{\sigma+i\theta})|
  &\geq
  \prod_{j=0}^{n-1}
  \left(
   A_1 \cparam^{2^{-(j+2)}} \delta^{-2^{-(j+1)}}
  \right)\\
  &=
  A_1^n
  \cparam^{\frac{1}{2}(1 - 2^{-n})}
  \delta^{-(1 - 2^{-n})}.
 \end{align*}
 and a similar upper bound.
 Finally, $\delta$ is given by
 \begin{align*}
  \delta
  &=
  | e^{\sigma + i\theta} - e^{i(\theta_n + \beta)} |\\
  &=
  | e^{\sigma + i\varphi} - 1 |\\
  &\asymp
  (\sigma^2 + \varphi^2)^{1/2},
 \end{align*}
 and so, modifying the constants as necessary,
 we have our result.
\end{proof}

We can now obtain our lower bound on the normalising factor.
\begin{proposition}
\label{thm:pf-bound}
 If $\nu > 2$, then
 there exists a constant $A$
 depending only on $\nu$
 such that for any fixed $T > 0$,
 for sufficiently small $\cparam$
 and for $n < \rdown{T/\cparam} \wedge \deadtime$,
 \begin{equation}
 \label{eq:pf-bound}
  Z_n
  \geq
  A^n
  \cparam^{\frac{\etasize}{2}(1 - 2^{-n})}
  \sigma^{-[\etasize(1 - 2^{-n}) - 1]}
 \end{equation} 
 provided that $\sigma = \sigma(\cparam) \leq
 L$.
\end{proposition}

\begin{proof}

 The normalising factor $Z_n$ is given by the integral
 $\int_{\T} |\Phi_n'(e^{\sigma + i\theta})|^{\etasize} \,\d\theta$,
 and Lemma \ref{thm:deriv-estimate} gives us a lower bound
 on the integrand for $\theta$ close to $\theta_n + \beta$:
 \[
  |\Phi_{n}'( e^{\sigma + i (\theta_n + \beta + \varphi)} )|^{\etasize}
  \geq
  A^n \cparam^{\frac{\etasize}{2}(1 - 2^{-n})}
  (\sigma^2 + \varphi^2)^{-\frac{\etasize}{2}(1 - 2^{-n})}
 \]
 when $|\varphi| < L$.
 
 We will now integrate our lower bound over the interval
 $(\theta_n + \beta - L, \theta_n + \beta + L)$.
 First, note that
 \begin{align*}
  \int_{-L}^{L}
   (\sigma^2 + \varphi^2)^{-\frac{\etasize}{2}(1 - 2^{-n})} \,
  \d\varphi
  &=
  \int_{-L/\sigma}^{L/\sigma}
   ( \sigma^2 + \sigma^2 x^2 )^{-\frac{\etasize}{2}(1 - 2^{-n})} \, \sigma
  \d x\\
  &=
  \sigma^{1 - \etasize (1 - 2^{-n})}
  \int_{-L/\sigma}^{L/\sigma}
   \frac{\d x}{(1 + x^2)^{\frac{\etasize}{2}(1 - 2^{-n})}}\\
  &\geq
  A' \sigma^{1 - \etasize (1 - 2^{-n})}
 \end{align*}
 for a constant $A'$,
 since the integral term on the right hand side is increasing as $\cparam \to 0$
 because $\sigma \ll L$.
 Note that this all remains true for any $\eta < 0$,
 and the fact that $\eta < -2$
 will only be necessary in \autoref{sec:concentration}.
 
 Finally, we can put together our bounds (and modify our constant $A$) to get
 \begin{align*}
  \int_{\theta_n + \beta - L}^{\theta_n + \beta + L}
   | \Phi_n'(e^{\sigma + i\theta}) |^{\etasize}
  \, \d \theta
  &\geq
  A^n \cparam^{\frac{\etasize}{2}(1 - 2^{-n})}
  \int_{-L}^{L}
   (\sigma^2 + \varphi^2)^{-\frac{\etasize}{2}(1 - 2^{-n})}
  \, \d \varphi\\
  &\geq
  A^n
  \cparam^{\frac{\etasize}{2}(1 - 2^{-n})}
  \sigma^{1 - \etasize (1 - 2^{-n})}
 \end{align*}
 as required. 
\end{proof}

\subsection{Concentration about each basepoint}
\label{sec:concentration}
Most of our upper bounds on $|\Phi_n'|$ will be established in \autoref{sec:small-contributions},
but we will find one here as it uses the
estimates from the previous section.
Using the terminology we introduce in
\autoref{sec:small-contributions}
and illustrate in \autoref{fig:categories},
in this section we look at \emph{singular points}
which are within $L$
of one of the ``main'' poles
$e^{i(\theta_n \pm \beta)}$
so the estimate of
Lemma \ref{thm:deriv-estimate}
is valid, but are
not within $D$ of these poles.

\begin{proposition}
\label{thm:concentration}
Let $n < \rdown{T/\cparam} \wedge \deadtime$.
For $\sigma(\cparam) \leq \cparam^{2^{2^{1/\cparam}}}$,
then with $L = \cparam^{2^{N+1}}$
and $D = \cparam^{9/2} \sigma^{1/2} \ll L$,
\[
\frac{1}{Z_n}
\int_{[-L, L] \setminus [-D, D]}
|\Phi_n'(e^{\sigma + i(\theta_n \pm \beta + \varphi)})|^{\nu}
\, \d \varphi
=
o(\cparam^\gamma)
\]
as $\cparam \to 0$,
for any constant $\gamma > 0$.
\end{proposition}

\begin{proof}
 Using the symmetry of our upper bound in Lemma \ref{thm:deriv-estimate},
 it will be enough to find the upper bound
 $\int_{D}^{L}
  |\Phi_{n}'(e^{\sigma + i(\theta_n+\beta+\varphi)})|^{\nu}
 \,\d\varphi
 \ll
 \cparam^\gamma Z_n$.
 We have, modifying the constant $A_2$ where necessary,
 \begin{align*}
  \int_{D}^{L}
  |\Phi_{n}'(e^{\sigma + i(\theta_n+\beta+\varphi)})|^{\nu}
 \,\d\varphi
 &\leq
 A_2^n
 \cparam^{\frac{\etasize}{2}(1 - 2^{-n})}
 \int_{D}^{L}
  (\sigma^2 + \varphi^2)^{-\frac{\etasize}{2}(1 - 2^{-n})}
 \,\d\varphi\\
 &=
 A_2^n
 \frac{
  \cparam^{\frac{\etasize}{2}(1 - 2^{-n})}
 }{
  \sigma^{\etasize(1 - 2^{-n}) - 1}
 }
 \int_{D/\sigma}^{L/\sigma}
  (1 + x^2)^{-\frac{\etasize}{2}(1 - 2^{-n})}
 \,\d x\\
 &\leq
 A_2^n
 \frac{
  \cparam^{\frac{\etasize}{2}(1 - 2^{-n})}
 }{
  \sigma^{\etasize(1 - 2^{-n}) - 1}
 }
 \int_{D/\sigma}^{L/\sigma}
  x^{-\etasize(1 - 2^{-n})}
 \,\d x\\
 &\leq
 A_2^n
 \frac{
  \cparam^{\frac{\etasize}{2}(1 - 2^{-n})}
 }{
  D^{\etasize(1 - 2^{-n}) - 1}
 },
 \end{align*}
 and so, using our lower bound on $Z_n$,
 \begin{align*}
  \frac{
   \int_{D}^{L}
    |\Phi_{n}'(e^{\sigma + i(\theta_n+\beta+\varphi)})|^{\nu}
   \,\d\varphi
  }{
   Z_n
  }
  &\leq
  (A_2/A)^n
  \left(
   \frac{\sigma}{D}
  \right)^{\etasize(1 - 2^{-n}) - 1}\\
  &=
  (A_2/A)^n
  \left(
   \cparam^{-9/2} \sigma^{1/2}
  \right)^{ \etasize(1 - 2^{-n}) - 1 }
 \end{align*}
 which, since $\etasize(1 - 2^{-n}) - 1 \geq \frac{1}{2}\etasize - 1 > 0$,
 decays faster than any power of $\cparam$
 as $\cparam \to 0$.
\end{proof}

Note that the above proof is the only place in which we use that $\eta < -2$.
If $-2 \leq \eta < 0$,
then $h_2$ achieves its maximum
around the two bases of the first particle,
but does not have strong concentration
around these points.
For $-2 < \eta < 0$
$h_2$ is still supported
on all of $\T$
as $\cparam \to 0$,
so there is no concentration.
If $\eta = -2$,
$h_2$ is supported only
nearby to $\theta_1 \pm \beta$,
but
the event
$D < |\theta_{2} - (\theta_1 \pm \beta)| \ll \beta$
retains a high probability as $\cparam \to 0$.
On this event, $\theta_2$
is not close enough
to $\theta_1 \pm \beta$
for our inductive arguments
in Proposition \ref{thm:sticky}
and Lemma \ref{thm:deriv-estimate} to apply.
We can no longer guarantee
that the poles of the second particle
are stronger than the older pole at the base of the first particle,
and so lose the SSRW-like behaviour of $(\theta_n)_{n \geq 1}$.
It then becomes extremely difficult to say how the process behaves,
but the scaling limit as $\cparam \to 0$ is unlikely to be described
by the Schramm--Loewner evolution.

\subsection{Symmetry of the two most recent basepoints}
\label{sec:symmetry}


There are two parts to the statement in Theorem \ref{thm:step-convergence}
about convergence of $h_{n+1}$ to the discrete measure
$\frac{1}{2}( \delta_{\theta_n - \beta} + \delta_{\theta_n + \beta} )$:
the previous two sections and \autoref{sec:small-contributions}
establish that $h_{n+1}$ is concentrated very tightly around $\theta_n \pm \beta$,
and we will show here that the weight given to each of these two points is
approximately equal.

\begin{remark}
\label{rmk:symmetry-proof-is-not-inductive}
Unlike the results from the previous two sections,
the following proposition is not inductive,
i.e.\ as long as $n < \rdown{T/\cparam} \wedge \deadtime$,
the density $h_{n+1}$ is approximately symmetric,
even if the choices of the previous angles were
not made symmetrically.
Even in the extreme case
where $(\theta_n)_{n \in \N}$ is close to an arithmetic progression:
$\theta_2 \approx \theta_1 + \beta$, $\theta_3 \approx \theta_2 + \beta,
\dotsc, \theta_n \approx \theta_{n-1} + \beta$,
we still have an almost symmetric $h_{n+1}$.
\end{remark}

\begin{proposition}
\label{thm:symmetry}
 Let $n < \rdown{T/\cparam} \wedge \deadtime$. Then
 \[
  \sup_{|\varphi| < D}
  \left|
  \log \left(
   \frac
   {
    \left|
     \Phi_{n}'\left(e^{\sigma + i(\theta_n + \beta + \varphi)}\right)
    \right|
   }
   {
    \left|
     \Phi_{n}'\left(e^{\sigma + i(\theta_n - \beta - \varphi)}\right)
    \right|
   }
  \right) \right|
  \leq
  A \cparam^{11/4}
 \]
 for some constant $A$ depending only on $T$.
\end{proposition}

\begin{proof}
Let $z_{\pm}
=
\exp\left( \sigma + i\left[\theta_n \pm (\beta + \varphi)\right] \right)$
for $|\varphi| < D$,
and write $\lambda_{\pm} = z_{\pm} - e^{i(\theta_n \pm \beta)}$.
We can then write
\begin{equation}
\label{eq:asymmetries}
\log 
  \left(
    \frac{
      | \Phi_n'(z_+) |
    }{
      | \Phi_n'(z_-) |
    }
  \right)
=
\sum_{j=0}^{n-1}
\log
  \left(
    \frac{
      | f_{n-j}'(\Phi_{n-j,n}(z_+)) |
    }{
      | f_{n-j}'(\Phi_{n-j,n}(z_-)) |
    }
  \right)
\end{equation}

and so we can estimate each term in \eqref{eq:asymmetries} separately.

The $j=0$ term is exactly 0, by the symmetry of $|f_n'|$ about $\theta_n$.

For $1 \leq j \leq n-1$,
we will use 
Lemma 4 of \cite{stv-ale}, 
which states that
$f'(z) = \frac{f(z)}{z} \frac{z-1}{(z-e^{i\beta})^{1/2}(z - e^{-i\beta})^{1/2}}$,
to compare the two derivatives in the $j$th term of \eqref{eq:asymmetries}.
Write $z_\pm^j = \Phi_{n-j,n}(z_\pm)$,
then the $j$th term in \eqref{eq:asymmetries} is
\begin{equation}
\label{eq:jth-symmetry-term}
|f_{n-j}'(z_\pm^j)|
=
\frac{
  | z_\pm^{j+1} |
}{
  | z_\pm^j |
}
\frac{
  | z_\pm^j - e^{i\theta_{n-j}} |
}{
  | z_\pm^j - e^{i\thetafake_{n-j+1}} |^{1/2}
  | z_\pm^j - e^{i\thetalegit_{n-j+1}} |^{1/2}
}
\end{equation}

There will be some telescoping in the product which allows us to find
\[
 \prod_{j=1}^{n-1} \frac{|z_{\pm}^{j+1}|}{|z_{\pm}^{j}|}
 =
 \frac{|z_{\pm}^n|}{|z_{\pm}^{1}|}.
\]

Then recall that in \autoref{sec:tracking} we derived estimates
for the distance of $z_{\pm}^n$ from $e^{i\thetalegit_{n-j+1}}$
in terms of $| \lambda_\pm |$.
So by Proposition \ref{thm:sticky}, as $e^{i\thetalegit_1} = 1$,
\[
 | z_\pm^{n} - 1 |
 =
 \left[\discoeff\right]^{2(1 - 2^{-n})}
 | \lambda_\pm |^{2^{-n}}
 (1 + O(\cparam^4))
 =
 O(\cparam^{17/4})
\]
since $|\lambda_\pm|^{2^{-n}} \lesssim D^{2^{-n}}
\ll L^{2^{-n}}
\leq \cparam^4$.
Therefore $|z_{\pm}^n| = 1 + O(\cparam^{17/4})$,
and similarly $|z_{\pm}^1| = 1 + O(\cparam^{17/4})$.\\

Having dealt with the first fraction in
all derivatives \eqref{eq:jth-symmetry-term}
at once, we will tackle the remaining terms
individually for each $1 \leq j \leq n-1$.

First note that by definition of $\thetalegit_{n-j+1}$,
$|e^{i\thetalegit_{n-j+1}} - e^{i\theta_{n-j}}| = |e^{i\beta} - 1|$.
Hence, using Proposition \ref{thm:sticky} again,
\begin{align*}
 | z_{\pm}^j - e^{i\theta_{n-j}} |
 &=
 | e^{i\thetalegit_{n-j+1}} - e^{i\theta_{n-j}} |
 \left[
  1 + O\left(
   \frac{
    | z_{\pm}^j - e^{i\thetalegit_{n-j+1}} |
   }{
    | e^{i\thetalegit_{n-j+1}} - e^{i\theta_{n-j}} |
   }
  \right)
 \right]\\
 &=
 |e^{i\beta} - 1|
 \left[
  1 + O\left(
   \cparam^{-2^{-(j+1)}} |\lambda_\pm|^{2^{-j}}
  \right)
 \right]\\
 &=
 |e^{i\beta} - 1|
 \left[
  1 + O\left(
   \cparam^{15/4}
  \right)
 \right]
\end{align*}
since $|\lambda_{\pm}|^{2^{-j}} \ll L^{2^{-(n-1)}} \leq \cparam^{4}$.

Similarly,
\begin{align*}
 | z_{\pm}^j - e^{i\thetafake_{n-j+1}} |
 &=
 |e^{2i\beta} - 1|
 (1 + O(\cparam^{15/4})),
\end{align*}
and finally, directly from Proposition \ref{thm:sticky},
\[
 |z_\pm^j - e^{i\thetalegit_{n-j+1}}|
 =
 \left[ \discoeff \right]^{2(1 - 2^{-j})}
 |\lambda_{\pm}|^{2^{-j}}
 (1 + O(\cparam^4)).
\]
Note that for the three estimates we just found,
the only part which depends on the choice of $\pm$
is the error term (as $|\lambda_+| = |\lambda_-|$).
Hence the part of the ratio of $|f_{n-j}'(z_{+}^j)|$
to $|f_{n-j}'(z_{-}^j)|$ which comes from the second
fraction in \eqref{eq:jth-symmetry-term}
is just $1 + O(\cparam^{15/4})$.

We can therefore find a constant $A$
(which does not depend on $n$ or $\varphi$)
such that for each $1 \leq j \leq n-1$,
$\left|
 \log\left(
  \frac{ |f_{n-j}'(z_+^j)| }{ |f_{n-j}'(z_-^j)| }
 \right)
\right|
\leq
A \cparam^{15/4}$.
As there are $O_T(\cparam^{-1})$ such terms
in the product \eqref{eq:asymmetries},
we have
\[
 \left|
  \log\left(
   \frac{
    | \Phi_n'(z_+) |
   }{
    | \Phi_n'(z_-) |
   }
  \right)
 \right|
 =
 O_T( \cparam^{11/4} )
\]
as claimed.
\end{proof}

Now we can deduce that $h_{n+1}$ gives (asymptotically)
the same measure to the sets $(\theta_n + \beta - D, \theta_n + \beta + D)$
and $(\theta_n - \beta - D, \theta_n - \beta + D)$.

\begin{remark}
Recall that earlier we used the heuristic argument that if $\eta = -\infty$
(so we choose from points
with the highest-order pole), then we attach the $(n+1)$th particle to one
of $\theta_n \pm \beta$, with equal probability.
With finite $\eta < -2$,
the derivative $|\Phi_n'|$ in fact differs slightly
at each of $e^{\sigma + i(\theta_n + \beta)}$
and $e^{\sigma + i(\theta_n - \beta)}$,
and so choosing to attach a particle at $e^{i\theta}$ for $\theta$ maximising
$|\Phi_n'(e^{\sigma + i\theta})|$
leads to a deterministic process
after the second step
rather than our SLE$_4$ limit.

However, when we have a finite $\eta < -2$,
integrating over the range $(-D,D)$ around each
$\theta_n \pm \beta$ means that only the asymptotic behaviour
of $|\Phi_n'|$ needs to be the same
to guarantee symmetry
between the two points $\theta_n \pm \beta$.
\end{remark}

\begin{corollary}
\label{thm:symmetry-integrated}
 For $n < \rdown{T/\cparam} \wedge \deadtime$,
 \begin{equation}
 \label{eq:symmetry-bound}
  \left|
   \int_{-D}^{D}
    h_{n+1}(\theta_n + \beta + \varphi)
   \,\d\varphi
   -
   \int_{-D}^{D}
    h_{n+1}(\theta_n - \beta - \varphi)
   \,\d\varphi
  \right|
  =
  O_T(\cparam^{11/4}).
 \end{equation}
\end{corollary}

\begin{proof}
 From Proposition \ref{thm:symmetry}, we have
 \begin{align*}
  \int_{-D}^{D}
   &h_{n+1}(\theta_n + \beta + \varphi)
  \,\d\varphi
  -
  \int_{-D}^{D}
   h_{n+1}(\theta_n - \beta - \varphi)
  \,\d\varphi\\
  &=
  \frac{1}{Z_n}
  \int_{-D}^{D}
   \left(
    | \Phi_n'(e^{\sigma + i(\theta_n + \beta + \varphi)}) |^{\nu}
    -
    | \Phi_n'(e^{\sigma + i(\theta_n - \beta - \varphi)}) |^{\nu}
   \right)
  \d\varphi\\
  &=
  \frac{1}{Z_n}
  \int_{-D}^{D}
   \left(
    | \Phi_n'(e^{\sigma + i(\theta_n + \beta + \varphi)}) |^{\nu}
    -
    e^{O_T(\cparam^{11/4})}
    | \Phi_n'(e^{\sigma + i(\theta_n + \beta + \varphi)}) |^{\nu}
   \right)
  \d\varphi\\
  &=
  O_T\left(
   \cparam^{11/4} \frac{
   \int_{-D}^{D}
    |\Phi_n'(e^{\sigma + i(\theta_n+\beta+\varphi)})|^{\nu}
   \,\d\varphi}{
   Z_n
   }
  \right)
 \end{align*} 
 which is just $O_T(\cparam^{11/4})$ by definition of $Z_n$.
\end{proof}

\section{Analysis of the density away from the main basepoints}
\label{sec:small-contributions}

In this section, we will classify the points $\theta \in \T$
with $| \theta - (\theta_n \pm \beta) | \geq D$
(i.e.\ the set $F_n$ from Theorem \ref{thm:step-convergence})
into \emph{regular points} $R_n$ where $h_{n+1}(\theta) \ll 1$,
and \emph{singular points} $S_n$ where $h_{n+1}(\theta) \gtrsim 1$.
We make this classification based on how close the image
$\Phi_n(e^{\sigma + i\theta})$ is to the common basepoint of the cluster,
which is the image of all the poles of $\Phi_n'$,
as we can see in \autoref{fig:categories}.

In \autoref{sec:wilderness} we make this classification explicit
and establish a bound on $h_{n+1}$ for the regular points.
In \autoref{sec:old-basepoints} we analyse the singular points
more carefully and establish an upper bound on
$\int_{S_n} h_{n+1}(\theta) \,\d\theta$
using similar techniques
as in \autoref{sec:partition-function}.

\begin{figure}[h]
\def\categoriesscale{0.4}
 \centering
 \begin{tikzpicture}
  \draw node (image) at (0,0) {
    \includegraphics[scale=\categoriesscale,
    angle=90,
    trim=
     0 90 0 0,
    clip
    ]
    {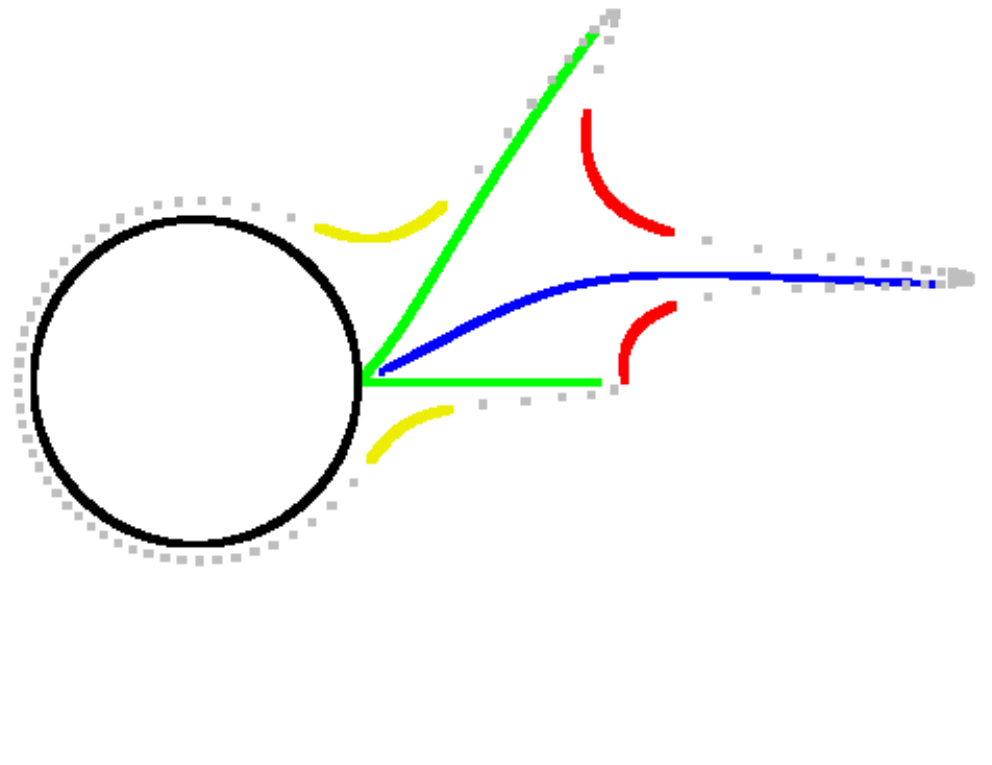}
  };
  \draw node (preimage) at
  (-15*\categoriesscale, 0)
  {
    \includegraphics[scale=\categoriesscale,
    angle=90,
    trim=
     0 90 0 90,
    clip
    ]
    {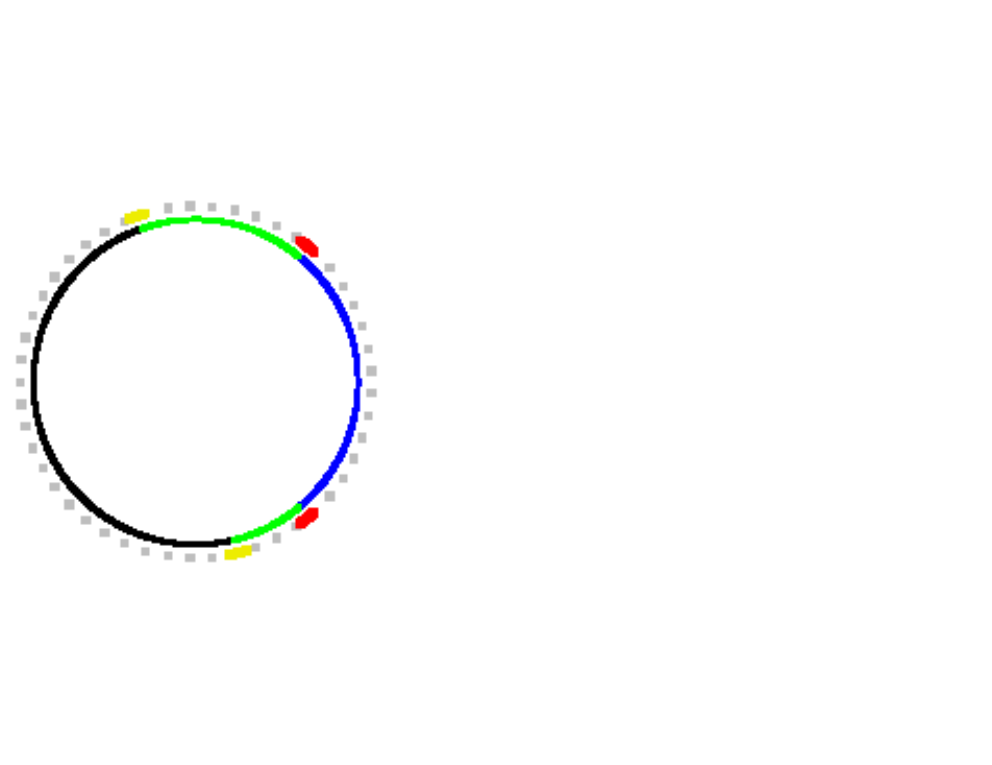}
  }; 
  \draw[->] ($(preimage.east)+(0,-5.5*\categoriesscale)$)
  --
  node [anchor=south] {$\Phi_3$}
  ($(image.west) + (3*\categoriesscale,-5.5*\categoriesscale)$);
 \end{tikzpicture}
 \caption{
\label{fig:categories} 
 We can see on the left the three types of points
 in $e^{\sigma}\T$ for the three-slit cluster:
 we have the \emph{singular points} in red and yellow
 and the \emph{regular points} in grey dots.
 The right hand side of the diagram shows that a point on $e^{\sigma} \T$
 is classified as regular
 if its image under $\Phi_n$ is far from the common basepoint
 (Proposition \ref{thm:distant-points}
 in \autoref{sec:wilderness} shows that this implies $h_{n+1} \ll 1$),
 and the singular points are further classified into the two main (red) arcs
 containing $e^{i(\theta_n \pm \beta)}$,
 and the other (yellow) singular points.
 We have $h_{n+1} \gtrsim 1$ for all singular points,
 but we obtained a lower bound on the integral of $|\Phi_n'|$ over the red regions
 in \autoref{sec:partition-function},
 and we will find an upper bound
 on the integral of this derivative over the yellow regions
 in \autoref{sec:old-basepoints}.
 Note that the choice of $\sigma$ we have used for this diagram
 is around $\cparam^2$ rather than the much smaller $\cparam^{2^{1/\cparam}}$,
 which is necessary to make the envelope $\Phi_3( e^{\sigma}\T )$ clear,
 but does mean that some ``regular'' points are closer to the common basepoints
 than the red ``singular'' points.
 With a sufficiently small $\sigma$ this isn't the case.
 }
\end{figure}

\subsection{Regular points}
\label{sec:wilderness}

In this section,
we will establish a criterion for $\theta \in \T$
to be in our set of \emph{regular} points for which $h_{n+1}(\theta) \ll 1$,
based on the position of $\Phi_n(e^{\sigma + i\theta})$,
as shown in \autoref{fig:categories}.

We will first derive an upper bound on $|\Phi_n'(w)|$
in terms of $|\Phi_n(w) - 1|$,
so we can classify $w \in \Delta$ as a regular point using the distance 
of its image $\Phi_n(w)$ from $1$.

\begin{proposition}
\label{thm:distant-points}
Let $n < N(\cparam) \wedge \deadtime$.
For $\theta \in \R$, let $w = \exp(\sigma + i\theta)$.

For any function $a : \R_+ \to \R_+$
with $D^{2^{-N}}/\beta
\leq a(\cparam) \leq \cparam^{3/2}$
for all $0 < \cparam < 1$,
if
\begin{equation}
 \label{eq:distant-points}
 |\Phi_n(w) - 
 1
 | 
 \geq
 \beta a(\cparam)
\end{equation}
then,
for sufficiently small $\cparam$,
\begin{equation}
\label{eq:wilderness-bound}
 | \Phi_n'(w) |
 \leq 
 A^n \beta^{n/2}
 \left(
  \frac{
   a(\cparam)
  }{
   8
  }
 \right)^{- \frac{1}{2} (2^n - 1) }
\end{equation}
where $A$ is a universal constant independent of $a$.
\end{proposition}

\begin{proof}

We will use the estimate \eqref{eq:distance-estimate}
from Lemma \ref{thm:distance-estimate}.
For convenience, let $z = \Phi_n(w)$,
and we will estimate $|\Phi_n'(w)| = |(\Phi_{n}^{-1})'(z)|^{-1}$
by using \eqref{eq:chain-rule-expansion} and estimating each term separately,
using Lemma \ref{thm:distance-estimate} to obtain estimates on
$\Phi_{n-j,n}(w) = \Phi_{n-j}^{-1}(z)$ by induction on $j$.

First we claim that
for $A(\cparam) \leq \cparam^{1/2}$,
and $\zeta \in \Delta \setminus (1, 1+d(\cparam)]$,
if we have
$|\zeta - 1| \geq \beta A(\cparam)$,
then
\begin{equation}
\label{eq:reverse-distance-estimate}
\min_\pm( | f^{-1}(\zeta) - e^{\pm i\beta} | )
\geq
\frac{1}{4} 
\beta A(\cparam)^2
\end{equation}
for all $\cparam < c_0$, where $c_0 > 0$ is a universal constant which doesn't depend on $A$.

To see this, suppose that
$| f^{-1}(\zeta) - e^{i\beta} | < \frac{1}{4} \beta A(\cparam)^2$.
Then by Lemma \ref{thm:distance-estimate},
setting $\eps = 2^{1/4} - 1 > 0$,
for sufficiently small $\cparam$,
\begin{align*}
| \zeta - 1 | &= | f(f^{-1}(\zeta)) - f(e^{i\beta}) |\\
&=
2(e^{\cparam} - 1)^{1/4} |f^{-1}(\zeta) - e^{i\beta}|^{1/2}
(1 + O\left( A(\cparam)^2 \vee \cparam^{1/2} A(\cparam) \right))\\
&<
2 (\beta/2)^{1/2} (1 + \eps)  \frac{1}{2} \beta^{1/2} A(\cparam) (1 + \eps)\\
&=
\beta A(\cparam),
\end{align*}
so we have shown the contrapositive for our claim.

The derivative $|\Phi_n'(w)|$ is decomposed
in \eqref{eq:chain-rule-expansion} into the product of $n$ terms
$\left| f'( e^{-i\theta_{k}} \Phi_{k,n}(w) ) \right|$,
and so we can find an upper bound on $|\Phi_n'(w)|$
by obtaining lower bounds on
each $| \Phi_{k,n}(w) - e^{i(\theta_{k} \pm \beta)} |
=
| \Phi_{k}^{-1}(z) - e^{i(\theta_{k} \pm \beta)} |$
for $0 \leq k \leq n-1$
and applying Lemma \ref{thm:f-prime-estimate}.

We claim that,
for each $0 \leq k \leq n-1$,
\begin{equation}
\label{eq:uniform-bound-induction}
| \Phi_{k}^{-1}(z) - e^{i\theta_{k+1}} |
\geq
\beta \times
8 \left(
 \frac{
 a(\cparam)
 }{
 8
 }
\right)^{2^k}
\end{equation}
and we will show this using induction.
For $k = 0$,
\eqref{eq:uniform-bound-induction}
is exactly the assumption \eqref{eq:distant-points} of this proposition.
For $k \geq 1$, we assume as the induction step that
\[
| \Phi_{k-1}^{-1}(z) - e^{i\theta_{k}} |
\geq
\beta \times 8 \left(
 \frac{
  a(\cparam)
 }{
  8
 }
\right)^{2^{k-1}}
\]
and aim to obtain \eqref{eq:uniform-bound-induction}
by applying \eqref{eq:reverse-distance-estimate}.

Taking $A(\cparam) = 8 \left(
 \frac{
  a(\cparam)
 }{
  8
 }
\right)^{2^{k-1}}$ in \eqref{eq:reverse-distance-estimate} gives us
\[
| \Phi_{k}^{-1}(z) - e^{i\thetalegit_{k+1}} |
\geq
\beta
\times 16 \left(
 \frac{
  a(\cparam)
 }{
  8
 }
\right)^{2^{k}},
\]
and so since $8 \beta \left(
 \frac{
  a(\cparam)
 }{
  8
 }
\right)^{2^{k}} \geq 2D$
when $k \leq
N \wedge \deadtime$
(for $\cparam$ sufficiently small),
\begin{align*}
| \Phi_k^{-1}(z) - e^{i\theta_{k+1}} |
&\geq
| \Phi_k^{-1}(z) - e^{i\thetalegit_{k+1}} |
- | e^{i\theta_{k+1}} - e^{i\thetalegit_{k+1}} |\\
&\geq
16 \beta \left(
 \frac{
  a(\cparam)
 }{
  8
 }
\right)^{2^{k}}
- 2D\\
&\geq 
8 \beta \left(
 \frac{
  a(\cparam)
 }{
  8
 }
\right)^{2^{k}},
\end{align*}
verifying \eqref{eq:uniform-bound-induction}.

Then \eqref{eq:uniform-bound-induction} tells us,
using \eqref{eq:reverse-distance-estimate},
that for each $0 \leq k \leq n-1$,
\begin{equation}
| \Phi_{k}^{-1}(z) - e^{i(\theta_{k} \pm \beta)} |
\geq
\beta \times 16
\left(
 \frac{
  a(\cparam)
 }{
  8
 }
\right)^{2^{k}},
\end{equation}
and so, by Lemma \ref{thm:f-prime-estimate},
for $\cparam$ sufficiently small,
\begin{align*}
| \Phi_n'(w) |
&=
\prod_{k=0}^{n-1} | f_{k+1}'( \Phi_{k}^{-1}(z) ) |\\
&\leq
A^n \beta^{n/2} \prod_{k=1}^{n-1}
\left(
 \beta^{1/2}
 \left[
  \beta \times 16 \left(
   \frac{
    a(\cparam)
   }{
    8
   }
  \right)^{2^{k}}
 \right]^{-1/2}
\right) \\
&=
(A/4)^n \beta^{n/2}
\left(
 \frac{
  a(\cparam)
 }{
  8
 }
\right)^{
 -\frac{1}{2} ( 2^n - 1 )
}
\end{align*}
for a universal constant $A$.
\end{proof}

In the next section we will use these results with $a(\cparam)$
equal to $\frac{L}{4\beta}$.
We can easily check now that if we use this choice of $a$
in Proposition \ref{thm:distant-points} then,
comparing \eqref{eq:wilderness-bound}
with \eqref{eq:pf-bound},
if $\sigma$ decays as fast as $\cparam^{2^{2^{N}}}$
then $|\Phi_n'(z)|^{\etasize}$
is far smaller than
$\cparam Z_n$,
for $z$ away from the preimages of $e^{i\theta_1}$,
and so if we classify our regular points as those $\theta$ for which
$|\Phi_n(e^{\sigma + i\theta}) - 1| \geq L/4$
then we do have $\sup\limits_{\theta \in R_n} h_{n+1}(\theta) \ll 1$.

\subsection{Old singular points}
\label{sec:old-basepoints}


In \autoref{sec:new-basepoints},
we established a lower bound
on the $n$th normalising factor $Z_n$.
So to show that the probability
is low that
the $(n+1)$th particle is attached
at a point in $E \subseteq \T$,
we need to find an upper bound on
$\int_{E} | \Phi_n'(e^{\sigma + i\theta})|^{\nu}\,\d\theta$.

We did this over certain regions
in \autoref{sec:wilderness}
by finding a bound
$|\Phi_n'(e^{\sigma + i\theta})|^{\nu}
\ll
\cparam Z_n
$.
In this section we will consider \emph{singular} points
where we can have
$|\Phi_n'(e^{\sigma + i\theta})|^{\nu}
\gg
Z_n$.
However, if we look at \autoref{fig:categories}
we can see that not all singular points are close to
the preimages $\theta_n \pm \beta$ of the base of the most recent particle;
there are singular points at the preimages of the base of each particle.
We will therefore need to estimate 
the integrand
$|\Phi_n'|^{\nu}$ more carefully,
and show that when integrated
over the singular points around these old bases
and normalised by $Z_n$,
the resulting probability is small.

The first thing we need to do is
to describe precisely which points
we are integrating over.
We have previously classified our points into regular points $R_n$
and singular points $S_n$ by looking at the distance $|\Phi_n(w) - 1|$.
Points are singular when $|\Phi_n(w) - 1| < \beta a(\cparam)$
(for an $a(\cparam)$ we will specify later),
and we will find a way of differentiating between the ``new'' singular points
around the preimages of the $n$th particle's base
and the ``older'' singular points around the preimages of the other
particles' bases.
To make this clear, we will first give names to all of these preimages.\\

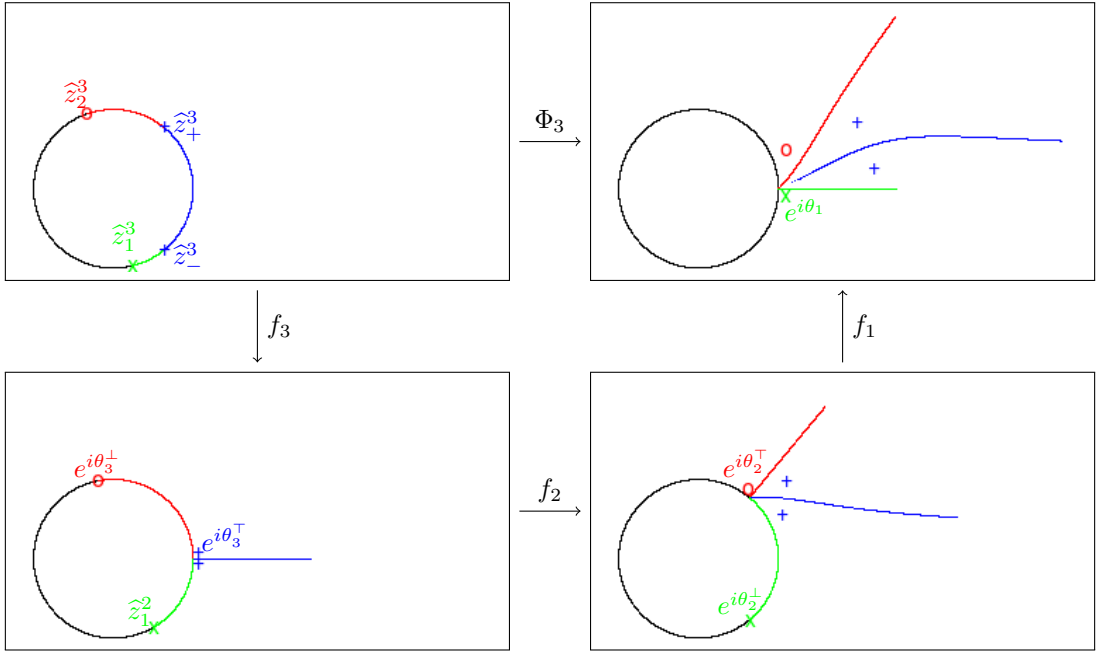
\begin{figure}[h] 
\def\mapscales{0.7}
 \centering
 \begin{tikzpicture}
  \draw node (p0) at (0*\mapscales,0*\mapscales) {
   \fbox{
    \includegraphics[scale=\mapscales,
    angle=0,
    trim=0 50 0 0,
  clip]{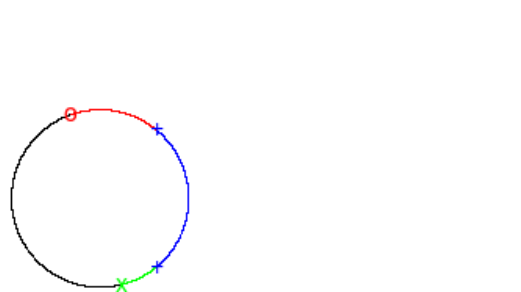}
   }
  };
  \draw node (p1) at (0*\mapscales,-7*\mapscales) {
   \fbox{
    \includegraphics[scale=\mapscales,
    angle=0,
    trim=0 50 0 0,
    clip]{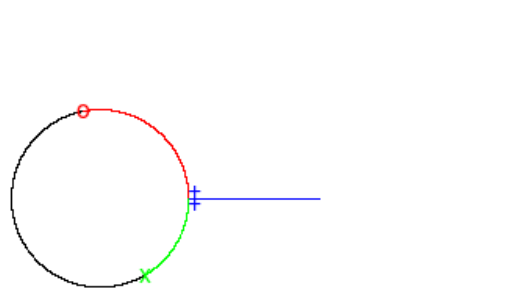}
   }
  };
  \draw node (p2) at (11*\mapscales,-7*\mapscales) {
  \fbox{
  \includegraphics[scale=\mapscales,
  angle=0,
  trim=0 50 0 0,
  clip]{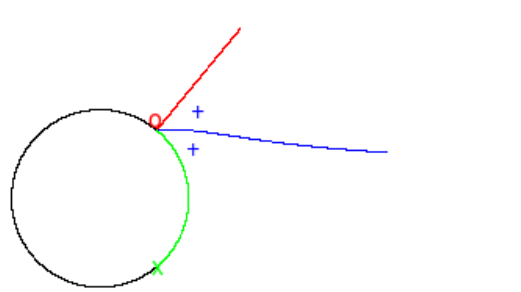}
  }
  };
  \draw node (p3) at (11*\mapscales,0*\mapscales) {
  \fbox{
  \includegraphics[scale=\mapscales,
  angle=0,
  trim=0 50 0 0,
  clip]{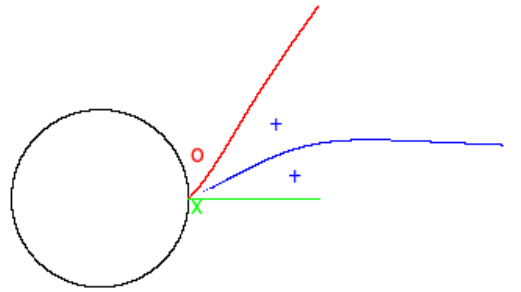}
  }
  };
  \draw[->] 
            (p0)
  -- node [anchor=west] {$f_3$}
            (p1);
  \draw[->] 
            (p1)
  -- node [anchor=south] {$f_2$}
            (p2);
  \draw[->] 
            (p2)
  -- node [anchor=west] {$f_1$}
            (p3);
  \draw[->] 
            (p0)
  -- node [anchor=south] {$\Phi_3$}
            (p3);
  \node[color = blue] (theta-minus) at (-1.3*\mapscales,-2.2*\mapscales) {$\bp_-^3$};
  \node[color=blue] (theta-plus) at (-1.3*\mapscales,0.3*\mapscales) {$\bp_+^3$};
  \node[color=red] (z-two-three) at (-3.4*\mapscales, 0.9*\mapscales) {$\bp_2^3$};
  \node[color=green] (z-one-three) at (-2.5*\mapscales, -1.8*\mapscales) {$\bp_1^3$};
  \node[color=blue] (theta-three) at (-0.6*\mapscales, -7.5*\mapscales) {$e^{i\thetalegit_3}$};
  \node[color=red]  (theta-three-fake) at (-3.0*\mapscales, -6.1*\mapscales) {$e^{i\thetafake_3}$};
  \node[color=green] (z-one-two) at (-2.2*\mapscales, -8.9*\mapscales) {$\bp_1^2$};
  \node[color=red] (theta-two) at (9.2*\mapscales, -6.1*\mapscales) {$e^{i\thetalegit_2}$};
  \node[color=green] (theta-two-fake) at (9.1*\mapscales, -8.7*\mapscales) {$e^{i\thetafake_2}$};
  \node[color=green] (theta-one) at (10.3*\mapscales,-1.3*\mapscales) {$e^{i\theta_1}$};
 \end{tikzpicture}
 \caption{
 \label{fig:composition}
  The construction of a cluster with three particles by composing the three maps
  $f_3$, $f_2$ and $f_1$.
  The top left diagram has labelled the four poles
  $\bp^3_\pm$, $\bp_2^3$ and $\bp_1^3$
  of $\Phi_3'$ with text, and
  the markers $+$, $\times$ and $\circ$
  have been used to track the images of $e^{\sigma} \bp$ for each
  pole $\bp$.
  By following the preimages of each point in the upper-right diagram
  through each map $f_1$, $f_2$ and $f_3$,
  we can see how we defined the ``lesser'' poles $\bp_2^3$ and $\bp_1^3$:
  for example, in the lower-right diagram $e^{i\thetafake_2}$ is a pole of $f_1'$,
  its preimage under $f_2$ is $\bp_1^2$, and
  the preimage of $\bp_1^2$ under $f_3$
  is $\bp_1^3$.
  Note that the three indicated intervals may overlap slightly,
  or have gaps between them,
  but these defects are too small to be seen
  in this diagram,
  and these $\bp$ points are well-defined in both the ``$\eta = -\infty$''
  case where the intervals coincide perfectly,
  and the case of finite $\eta < -2$.
 }
\end{figure}

Firstly, we have the two ``most attractive'' points:
the preimages of the base of the most recent ($n$th) slit.
We will call these two points $\bp^n_\pm = e^{i(\theta_n \pm \beta)}$.
Now the other points correspond to the bases of the $n-1$ other slits
in the cluster,
and we will denote them by $\bp^n_j$ for $1 \leq j \leq n-1$.
The base of the first slit is the image under $f_1$
of the choice of $e^{i(\theta_2 \pm \beta)}$
which is \emph{not} close to $e^{i\theta_2}$.
We defined this
in Definition \ref{def:thetafake}
to be $e^{i\thetafake_2}$,
and so the point sent to the base of the first slit by $\Phi_n$
is the preimage under $f_2 \circ \dotsb \circ f_n = \Phi_{1,n}$
of $e^{i\thetafake_2}$,
so set $\bp^n_1 = \Phi_{1,n}^{-1}(e^{i\thetafake_2})$.

In general, when the $j$th slit is attached to the cluster by
$f_j$, there are two points which are mapped to the base of the slit:
$e^{i\thetalegit_{j+1}}$ (where the later slits are also attached),
and $e^{i\thetafake_{j+1}}$, which has nothing else attached to it.
Therefore, the point sent to the base of the $j$th slit by $\Phi_n$
is the preimage of $e^{i\thetafake_{j+1}}$
under $f_{j+1} \circ \dotsb \circ f_n$.
We can see this illustrated in \autoref{fig:composition}.

\begin{definition}
\label{def:basepoints}
 The base of the $j$th slit for $1 \leq j \leq n-1$ is the image of
 \begin{equation}
 \label{eq:basepoints}
  \bp^n_j := \Phi_{j,n}^{-1}\left( e^{i\thetafake_{j+1}} \right)
 \end{equation}
 under $\Phi_n$.
\end{definition}

Note that for all $n < N \wedge \deadtime$ and
$1 \leq j \leq n-1$,
\begin{equation}
\label{eq:inductive-bp}
 f_n(\bp^n_j) = \bp^{n-1}_j,
\end{equation}
where we adopt the convention that $\bp^{n-1}_{n-1} = e^{i\thetafake_{n}}$.

\begin{remark}
 We will bound $|\Phi_n'(w)|$ above when $w$ is close to $\bp_j^n$,
 so first we will have to show that these points $\bp_j^n$
 for $1 \leq j \leq n-1$ are not close to the points $e^{i(\theta_n \pm \beta)}$
 where we have already shown $|\Phi_n'|$ is large.
\end{remark}

\begin{lemma}
\label{thm:basepoint-separation}
 For $n < N \wedge \deadtime$ and $1 \leq j \leq n-1$,
 \[
  | e^{i(\theta_n \pm \beta)} - \bp_j^n |
  \geq
  \cparam^{2^{n-j}},
 \]
 when $\cparam$ is sufficiently small.
\end{lemma}

\begin{proof}
 Assume for contradiction that
 $| e^{i(\theta_n + \beta)} - \bp_j^n | < \cparam^{2^{n-j}}$.
 By Lemma \ref{thm:distance-estimate},
 \begin{align*}
  | e^{i\theta_n} - \bp_j^{n-1} |
  &=
  | f_n(e^{i(\theta_n + \beta)}) - f_n(\bp_j^n) |\\
  &=
  2(e^\cparam - 1)^{1/4} \cparam^{2^{n-j-1}}
  \left( 1 + O\left(\cparam^{1/4}\cparam^{2^{n-j-1}}\right) \right)\\
  &<
  \frac{1}{2}\cparam^{2^{n-j-1}}
 \end{align*}
 for $\cparam$ smaller than some universal $c_0$ (with $(c_0 - 1)^{1/4} < 1/4$,
 and small enough to make the error term irrelevant),
 and so
 \begin{align}
 \label{eq:bp-inductive-distance}
  | e^{i\thetalegit_n} - \bp_j^{n-1} |
  \leq
  | e^{i\thetalegit_n} - e^{i\theta_n} | + | e^{i\theta_n} - \bp_j^{n-1} |
  <
  \cparam^{2^{n-j-1}},
 \end{align}
 since $|e^{i\thetalegit_n} - e^{i\theta_n}| \lesssim D \ll \cparam^{2^{n-j-1}}$.
 Then, as $\thetalegit_n = \theta_{n-1} \pm \beta$ for some choice of $\pm$,
 we can apply this argument repeatedly until we arrive at
 $| e^{i\thetalegit_{j+1}} - \bp_j^j | < \cparam^{2^{j - j}} = \cparam$.
 But as we noted after \eqref{eq:inductive-bp},
 $\bp_j^j = e^{i\thetafake_{j+1}}$,
 and $| e^{i\thetalegit_{j+1}} - e^{i\thetafake_{j+1}} |
 \sim 4\cparam^{1/2} \gg \cparam$,
 and so we have our contradiction.
\end{proof}

\begin{remark}
 In fact the lower bound in
 Lemma \ref{thm:basepoint-separation}
 is fairly generous;
 it would take only a small amount of extra work
 in the proof above to get a tighter bound
 of $\cparam^{2^{n-j-1}}$,
 and we could improve this even further as we used the weak bound
 $(e^\cparam - 1)^{1/4} < \frac{1}{4}$ in the initial calculation.
 However, all we need from Lemma \ref{thm:basepoint-separation}
 is a bound which decays more slowly than $L = \cparam^{2^{N+1}}$,
 and so we have chosen the bound which leads to the simplest possible proof.
\end{remark}

\begin{remark}
 The following corollary
 (which we will not prove)
 is not used in the
 proof of our main results,
 but does answer a question we may worry about:
 if we know that $w$ is within $L$ of some $\bp_j^n$,
 then is that $j$ uniquely determined?
\end{remark}

\begin{corollary} 
\label{thm:separation-corollary}
 For $n < N \wedge \deadtime$, if $1 \leq j < k \leq n-1$, then
 \[
  | \bp_j^n - \bp_k^n | \geq \cparam^{2^{n-j}}
 \]
 for sufficiently small $\cparam$.
\end{corollary}

\begin{remark}
 The next result will be useful in telling us for which points $\theta \in \T$
 we can bound $|\Phi_n'(e^{\sigma+i\theta})|$ above
 using Proposition \ref{thm:distant-points},
 and will later help us locate those points for which
 Proposition \ref{thm:distant-points}
 does not provide an upper bound.
\end{remark}

\begin{lemma}
\label{thm:close-definition}
 Suppose that $n < N \wedge \deadtime$, and let $w \in \Delta$.
 For all $\cparam$ sufficiently small,
 if $| \Phi_n(w) - 1 |
 \leq
 \frac{L}{4}$,
 then either
 $\min\limits_{\pm} | w - e^{i(\theta_n \pm \beta)} | \leq L$,
 or
 there exists some $1 \leq j \leq n-1$ such that
 \[
  | \Phi_{j,n}(w) - e^{i\thetafake_{j+1}} |
  \leq
  \frac{\beta}{4}
  \left(
   \frac{
    L
   }{
    \beta
   }
  \right)^{2^{j}}.
 \]
\end{lemma}

\begin{proof}
 Suppose that there is no such $j$.
 We will show that $\min_\pm |w - e^{i(\theta_n\pm\beta)}| \leq L$.
 We now claim that $|\Phi_{j,n}(w) - e^{i\thetalegit_{j+1}}| \leq \frac{\beta}{4}
 \left( \frac{L}{\beta} \right)^{2^{j}}$
 for all $0 \leq j \leq n-1$
 (where $\Phi_{0,n} = \Phi_n$
 and $\thetalegit_1 = \theta_1 = 0$).
 For $j=0$ the claim is the true by assumption,
 and if the claim is true for $0 \leq j < n-1$,
 then by Lemma \ref{thm:inverse-distance-estimate},
 as $|\Phi_{j,n} - e^{i\theta_{j+1}}|
 \leq \frac{\beta}{4}\left(\frac{L}{\beta}\right)^{2^{j}}
 + |e^{i\thetalegit_{j+1}} - e^{i\theta_{j+1}}|
 \leq \frac{\beta}{2}\left(\frac{L}{\beta}\right)^{2^{j}}$,
 for sufficiently small $\cparam$,
 \begin{align*}
  \min(
  | \Phi_{j+1,n}(w) - e^{i\thetalegit_{j+2}} |,
  | \Phi_{j+1,n}(w) - e^{i\thetafake_{j+2}} |)
  &\leq
  \frac{
   \frac{1}{4}\beta^2 \left( \frac{L}{\beta}\right)^{2^{j+1}}
  }{
   4 (e^{\cparam} - 1)^{1/2}
  }
  \left(1 + \frac{1}{2}\right)\\
  &=
  \frac{3\beta/2}{4(e^{\cparam} - 1)^{1/2}}\times
  \frac{\beta}{4} \left( \frac{L}{\beta} \right)^{2^{j}}\\
  &\leq
  \frac{\beta}{4} \left( \frac{L}{\beta} \right)^{2^{j}}
 \end{align*}
 since $\beta \sim 2(e^{\cparam} - 1)^{1/2}$ for small $\cparam$.
 But we supposed initially that
 $| \Phi_{j+1,n}(w) - e^{i\thetafake_{j+2}} | > \frac{\beta}{4} \left( \frac{L}{\beta} \right)^{2^{j+1}}$,
 and so the above shows that
 $| \Phi_{j+1,n}(w) - e^{i\thetalegit_{j+2}} | \leq \Lj{j+1}$,
 and by induction our claim holds.
 Finally, one more application of Lemma \ref{thm:inverse-distance-estimate}
 after the $j = n-1$ case of our claim,
 $| \Phi_{n-1,n}(w) - e^{i\theta_n}| \leq \frac{\beta}{2}
 \left( \frac{L}{\beta} \right)^{2^{n-1}}$, tells us that
 $\min\limits_{\pm} | w - e^{i(\theta_n \pm \beta)} |
 \leq
 \frac{3\beta/2}{16(e^{\cparam}-1)^{1/2}}\beta \left( \frac{L}{\beta} \right)^{2^n}
 \ll L$,
 as required.
\end{proof}

\begin{remark}
 We intend to use this lemma to find a precise expression for our
 set $S_n$ of singular points
 and then we can make a precise estimate on the size of
 $|\Phi_n'(e^{\sigma + i\theta})|$ for $\theta \in S_n$
 as we did in Lemma \ref{thm:deriv-estimate}.
 For a singular point $w$, Lemma \ref{thm:close-definition}
 tells us that for some $j$, $\Phi_{j,n}(w)$ is close to
 $e^{i\thetafake_{j+1}}$,
 and we now need to turn that into an estimate for the distance between
 $w$ and $\Phi_{j,n}^{-1}( e^{i\thetafake_{j+1}} ) = \bp_j^n$.

\end{remark}

\begin{corollary}
\label{thm:point-locations}
 Suppose that $n < N \wedge \deadtime$,
 and let $w \in \Delta$.
 For all $\cparam$ sufficiently small,
 if $|\Phi_n(w) - 1| \leq \frac{L}{4}$
 then either $\min_{\pm} | w - e^{i(\theta_n \pm \beta)} | \leq L$
 or there exists some $1 \leq j \leq n-1$ such that
 \[
  | w - \bp_j^n |
  \leq
  A^{n-j}
  \Lj{j},
 \]
 where $A$ is some universal constant.
\end{corollary}

\begin{proof}
 To deduce this from Lemma \ref{thm:close-definition},
 we need only show that there is some constant $A$ such that
 $| \Phi_{j,n}(w) - e^{i\thetafake_{j+1}} |
 \leq
 \Lj{j}
 \implies
 | w - \bp_j^n |
 \leq
 A^{n-j} \Lj{j}$.
 Fix some $1 \leq j \leq n-1$.
 We will show that for $j \leq k \leq n-1$,
 $|\Phi_{k+1,n}(w) - \bp_j^{k+1}| \leq A | \Phi_{k,n}(w) - \bp_j^{k}|$.\\
 
 Fix a path $\gamma : (0,1] \to \Delta$
 with $\lim_{\eps \downarrow 0}\gamma(\eps) = \bp_j^k$,
 $\gamma(1) = \Phi_{k,n}(w)$,
 and $|\gamma(t) - \bp_j^k| \leq |\Phi_{k,n}(w) - \bp_j^k|$
 for all $t \in (0,1]$.
 We can also choose $\gamma$ in such a way that it has arc length
 $\ell := \int_\gamma |\d z| \leq 2|\Phi_{k,n}(w) - \bp_{j}^k|$.
 By the fundamental theorem of calculus,
 \begin{align*}
  | \Phi_{k+1,n}(w) - \bp_j^{k+1} |
  &=
  | f_{k+1}^{-1}( \Phi_{k,n}(w) ) - f_{k+1}^{-1}( \bp_j^{k} ) |\\
  &=
  \left|
   \int_\gamma
    (f_{k+1}^{-1})'(\zeta)
   \,\d \zeta
  \right|\\
  &\leq
  \ell \times \sup_{\zeta \in \gamma(0,1]} | (f_{k+1}^{-1})'(\zeta) |\\
  &=
  \frac{\ell}{\inf_{\omega \in f_{k+1}^{-1}(\gamma(0,1])} |f_{k+1}'(\omega)|}.
 \end{align*}
 Now
 there must be some constant $M \geq 1$ such that
 $|\omega - e^{i\theta_{k+1}}| \geq \beta/M$
 for all $\omega \in f_{k+1}^{-1}(\gamma(0,1])$.
 Otherwise, if $|\omega - e^{i\theta_{k+1}}| < \beta/M$,
 then it is easy to check using the explicit form of $f_\cparam$
 from \cite{marshall-rohde}
 that $| f_{k+1}(\omega) - e^{i\theta_{k+1}}(1 + d) | = O(\beta/M^2)$,
 and so
 \[|\bp_j^k - e^{i\theta_{k+1}}(1+d)|
 \leq
 |f_{k+1}(\omega) - e^{i\theta_{k+1}}(1+d)| + |f_{k+1}(\omega) - \bp_j^k|
 \leq
 \frac{1}{2}d\]
 for sufficiently large $M$, 
 contradicting $|\bp_j^k| = 1$.
 Hence by Lemma \ref{thm:f-prime-estimate},
 there is a constant $A$ such that
 \[\inf\limits_{\omega \in f_{k+1}^{-1}(\gamma(0,1])} |f_{k+1}'(\omega)| \geq 2A^{-1}.\]
 
 We therefore obtain
 \begin{equation}
 \label{eq:inductive-old-bp-distance}
  | \Phi_{k+1,n}(w) - \bp_j^{k+1} | \leq A | \Phi_{k,n}(w) - \bp_j^{k} |
 \end{equation}
 for all $j \leq k \leq n-1$,
 and so
 \[
  |w - \bp_j^n|
  =
  |\Phi_{n,n}(w) - \bp_j^n|
  \leq
  A^{n-j} |\Phi_{j,n}(w) - \bp_j^j|
  \leq
  A^{n-j} \Lj{j},
 \]
 as required.
\end{proof}

%
%
If we let $L_j^n$ be the upper bound
in Corollary \ref{thm:point-locations},
then we now have a necessary condition
for points to be singular,
based only on their location:
if $e^{\sigma + i\theta} \in \Delta$
is not within $L_j^n$ of $\bp_j^n$
for some $j$,
then $\theta$ is regular.
The set of singular points $S_n$
is therefore contained in the union
of only $n+1$ intervals centred
around $e^{\theta_n \pm \beta}$
and each $\bp_j^n$.

We can now find a precise estimate for $|\Phi_n'|$
on $S_n$ as we did in Lemma \ref{thm:deriv-estimate}.
The proof will also be similar to that of Lemma \ref{thm:deriv-estimate}.

\begin{lemma}
\label{thm:old-bp-bound}
 Let $n < N \wedge \deadtime$,
 and $1 \leq j \leq n-1$.
 If $\cparam$ is sufficiently small,
 then for all $w \in \Delta$ with $|w| = e^{\sigma}$
 and
 $| w - \bp_j^n | \leq A^{n-j} \Lj{j}$,
 for $A$ as in Corollary \ref{thm:point-locations},
 we have
 \[
  | \Phi_n'( w ) |
  \leq
  B^n
  \cparam^{\frac{n-j}{4} + 1}
  \cparam^{\frac{1}{2}(1 - 2^{-j})}
  \frac{1}{
   \cparam^{2^{n-j}}
  }
  | w - \bp_j^n |^{-(1 - 2^{-j})}
 \]
 where $B$ is a universal constant.
\end{lemma}

\begin{proof}
 We will complete the proof by finding bounds on
 $|\Phi_{j,n}(w) - e^{i\thetafake_{j+1}}|$;
 an upper bound to show $|\Phi_{j,n}'(w)|$ is small,
 and a lower bound to show $| \Phi_j'(\Phi_{j,n}(w))|$ is small.
 The rest of the proof will be similar to the way we deduced
 Lemma \ref{thm:deriv-estimate}
 from
 Proposition \ref{thm:sticky}.
 
 First, we will estimate the positions of
 $\Phi_{n-1,n}(w), \Phi_{n-2,n}(w), \dotsc, \Phi_{j,n}(w)$.
 As in the proof of Corollary \ref{thm:point-locations},
 for $j+1 \leq k \leq n$,
 \begin{align}
  \nonumber
  | \Phi_{k-1,n}(w) - \bp_j^{k-1} |
  &=
  | f_k( \Phi_{k,n}(w) ) - f_k( \bp_j^k ) |\\
  \label{eq:path-integral-upper-bound}
  &\leq
  2 | \Phi_{k,n}(w) - \bp_j^k |
  \times
  \sup_{|\zeta - \bp_j^k| \leq |\Phi_{k,n}(w) - \bp_j^k|}| f_k'(\zeta) |,
 \end{align}
 so we need only bound $| f_k'(\zeta) |$
 for $\zeta$ close to $\bp_j^k$.
 We will also need inductively that $| \Phi_{k,n}(w) - \bp_j^k |$
 is small in order to say that $\zeta$ \emph{is} close to $\bp_j^k$.
 \begin{claim}
  For $j+1 \leq k \leq n$,
  $| \Phi_{k,n}(w) - \bp_j^k |
  \leq
  A^{n-j}
  \cparam^{3 \times 2^n}
  $
  for sufficiently small $\cparam$.
 \end{claim}
 The claim is true for $k = n$,
 as $|w - \bp_j^n|
 \leq
 A^{n-j} \cparam^{1/2} \left(\frac{1}{2}\cparam^{2^{n+1} - 1/2}\right)^{2^j}
 \leq
 A^{n-j} \cparam^{2^{n+j+1} - 2^{j-1}}
 \leq
 A^{n-j} \cparam^{2^{n+2} - 2^{n}}$.
 Then, if the claim holds for all $l \geq k$,
 we have
 \begin{align*}
  | \Phi_{l,n}(w) - \bp_j^l |
  &\leq
  A^{n-j} \cparam^{3 \times 2^n}
  \leq
  \frac{1}{2}
  \cparam^{2^{l-j}}
 \end{align*}
 for all sufficiently small $\cparam$,
 and so, by Lemma \ref{thm:basepoint-separation} and the triangle inequality,
 for all $\zeta$ such that $|\zeta - \bp_j^l| \leq |\Phi_{l,n}(w) - \bp_{j}^l|$,
 we have
 $\min_{\pm}|\zeta - e^{i( \theta_l \pm \beta )}| \geq \frac{1}{2}\cparam^{2^{l-j}}$.
 Hence by
 Lemma \ref{thm:f-prime-estimate},
 \[
  | f_k'(\zeta) |
  \leq
  A_2 \frac{\cparam^{1/2}}{\cparam^{2^{l-j-1}}}
 \]
 Therefore, by \eqref{eq:path-integral-upper-bound},
 \begin{align*}
  | \Phi_{k-1,n}(w) - \bp_j^{k-1} |
  &\leq
  2^{n-k+1} | \Phi_{n,n}(w) - \bp_j^n |
  \times
  \prod_{l=k}^{n} \left(
   A_2 \cparam^{\frac{1}{2} - 2^{l-j-1}}
  \right)\\
  &\leq
  (2A_2)^{n-k+1} A^{n-j}\Lj{j}
  \cparam^{\frac{n-k+1}{4}}
  \cparam^{-\sum_{l=k-j-1}^{n-j-1} 2^{l}}\\
  &\leq
  \left[
   (2A_2)^{n-k+1}
   \cparam^{\frac{n-k+1}{4}}
  \right]
  A^{n-j}
  \left(\cparam^{2^{n+1} - \frac{1}{2}} \right)^{2^j}
  \cparam^{- ( 2^{n-j} - 2^{k-j-1} ) }\\
  &\leq
  A^{n-j}
  \cparam^{2^{n+j+1} - 2^{j-1} - 2^{n-j} + 2^{k-j-1}}\\
  &\leq
  A^{n-j}
  \cparam^{2^{n+2} - 2^{n-1} - 2^{n-1}}\\
  &=
  A^{n-j}
  \cparam^{3 \times 2^n},
 \end{align*}
 and so our claim holds by induction.
 
 We can also see, from the same computation,
 that
 \begin{equation}
 \label{eq:bp-distance-upper-bound}
  |\Phi_{j,n}(w) - e^{i\thetafake_{j+1}}| = |\Phi_{j,n}(w) - \bp_j^j | \leq \cparam^{3 \times 2^{n}}.
 \end{equation}
 Then for each $j+1 \leq k \leq n$,
 as $\cparam^{3 \times 2^n} \leq \frac{1}{2} \cparam^{2^{k-j}}$,
 we have by the triangle inequality and Lemma \ref{thm:basepoint-separation}
 that
 $| \Phi_{k,n}(w) - e^{i(\theta_k \pm \beta)} | \geq \frac{1}{2} \cparam^{2^{k-j}}$,
 and so by Lemma \ref{thm:f-prime-estimate},
 \begin{align}
 \nonumber
  |\Phi_{j,n}'(w)|
  &=
  \prod_{k=j+1}^n
   | f_k'( \Phi_{k,n}(w) ) |
  \\
 \nonumber
  &\leq
  \prod_{k=j+1}^n
   A_2 \frac{\beta^{1/2}}{(\frac{1}{2}\cparam^{2^{k-j}})^{1/2}}
  \\
 \nonumber
  &\leq
  (2A_2)^{n-j}
  \cparam^{
   \frac{n-j}{4} - \sum_{k=0}^{n-j-1} 2^k
  }\\
 \label{eq:away-from-pole-bound}
  &=
  (2A_2)^{n-j}
  \cparam^{
   \frac{n-j}{4}
   - 2^{n-j}
   + 1
  }
 \end{align}
 for sufficiently small $\cparam$.\\
 
 We will next establish an upper bound on $|\Phi_j'( \Phi_{j,n}(w) )|$.
 By the arguments used to prove Corollary \ref{thm:point-locations},
 we have a lower bound on $| \Phi_{j,n}(w) - e^{i\thetafake_{j+1}}|$
 as well as the upper bound we just established:
 \begin{equation}
 \label{eq:bp-distance-lower-bound}
  | \Phi_{j,n}(w) - e^{i\thetafake_{j+1}} | \geq A^{-(n-j)} | w - \bp_j^n |,
 \end{equation}
 where $A$ is a constant.
 The upper bound in \eqref{eq:bp-distance-upper-bound}
 is less than $\cparam^{2^{n + 1}}$,
 and so we can apply (the proof of) Lemma \ref{thm:deriv-estimate}
 to say
 \begin{align}
 \label{eq:near-pole-bound}
  | \Phi_{j}'( \Phi_{j,n}(w) ) |
  \leq
  \frac{(A')^j}{A^{\frac{n-j}{2}}} \frac{
   \cparam^{\frac{1}{2} (1 - 2^{-j})}
   }{
    |w - \bp_j^n|^{1 - 2^{-j}}
   },
 \end{align}
 and so we can combine \eqref{eq:away-from-pole-bound}
 and \eqref{eq:near-pole-bound}
 to obtain
 \begin{align*}
  | \Phi_n'(w) | &= | \Phi_{j,n}'(w) | \times | \Phi_j'( \Phi_{j,n}(w) ) |\\
  &\leq
  \left( \frac{2A_2}{\sqrt{A}} \right)^{n-j}
  ( A' )^{j}
  \cparam^{\frac{n-j}{4} - 2^{n-j} + 1}
  \cparam^{\frac{1}{2}(1 - 2^{-j})}
  | w - \bp_j^n |^{-(1-2^{-j})}
  \\
  &\leq
  (A'')^n
  \cparam^{\frac{n-j}{4} + 1}
  \cparam^{\frac{1}{2}(1 - 2^{-j})}
  \frac{1}{
   \cparam^{2^{n-j}}
  }
  | w - \bp_j^n |^{-(1 - 2^{-j})}
 \end{align*}
 where $A'' = \max( \frac{2A_2}{\sqrt{A}}, A' )$ is a constant.
\end{proof}


\begin{corollary}
\label{thm:old-bp-probability-bound}
 Let $n < N \wedge \deadtime$,
 and $1 \leq j \leq n-1$.
 Then for $L_j^n = A^{n-j} \Lj{j}$,
 we have
 \begin{align}
 \label{eq:old-bp-probability-bound}
  \int_{-L_j^n}^{L_j^n}
   | \Phi_n'( \bp_j^n e^{\sigma + i \varphi} ) |^{\nu}
  \,\d\varphi
  \leq
  B_{\etasize}^{n} \frac{
   \cparam^{\etasize( \frac{n-j}{4} + 1 )}
  }
  {
   \cparam^{\etasize 2^{n-j}}
  }
  \cparam^{\frac{\etasize}{2}( 1 - 2^{-j} )}
  \sigma^{-\left[ \etasize(1 - 2^{-j}) - 1 \right]}
 \end{align}
 where $B_{\etasize}$ is a constant depending only on $\etasize$.
\end{corollary}

\begin{proof}
 As $| \bp_{j}^n e^{\sigma + i\varphi} - \bp_j^n |
 \asymp
 ( \sigma^2 + \varphi^2 )^{1/2}$,
 the bound follows immediately from Lemma \ref{thm:old-bp-bound}
 (in the same way as we obtained Proposition \ref{thm:pf-bound}
 from Lemma \ref{thm:deriv-estimate}).
\end{proof}

\section{Proof of main results}
\label{sec:final-proof}

With the results of the previous sections,
we are finally ready to prove our main scaling limit result,
that the cluster $K_N^\cparam$ converges in distribution,
as $\cparam \to 0$,
to an SLE$_4$ cluster.
To help picture the sets $S_{n,j}$ and $R_n$,
it may be useful to refer to \autoref{fig:categories}.

\begin{proof}[Proof of Theorem \ref{thm:step-convergence}]
 We want to show that
 $h_{n+1}(F_n) = \int_{F_n} h_{n+1}(\theta) \,\d\theta$ is small,
 and so we will decompose $F_n$ into several sets.

 Let $R_n = \{ \theta \in \T :
 | \Phi_n(e^{\sigma + i\theta}) - 1 |
 >
 \frac{L}{4} \}$,
 $S_n = F_n \setminus R_n$.
 We will further decompose $S_n$:
 first define $$T_n = \{ \theta \in S_n :
 D < \min_{\pm}| e^{\sigma + i\theta} - e^{i(\theta_n \pm \beta)} | \leq L \},$$
 and for $1 \leq j \leq n-1$
 define $$S_{n,j} = \{ \theta \in S_n :
 | e^{\sigma + i\theta} - \bp_{j}^{n} | \leq L_j^n \},$$
 where $L_j^n$ is the bound appearing in Corollary \ref{thm:point-locations},
 then Corollary \ref{thm:point-locations} tells us that
 $S_n = T_n \cup \left(\bigcup_{j=1}^{n-1} S_{n,j}\right)$.
 We can then split the integral as
 \begin{equation}
 \label{eq:integral-split}
  h_{n+1}(F_n)
  \leq
  h_{n+1}(R_n) + h_{n+1}(T_n) + \sum_{j=1}^{n-1} h_{n+1}(S_{n,j}).
 \end{equation}
 We showed in \autoref{sec:concentration} that $h_{n+1}(T_n) = o(\cparam^\gamma)$
 for any fixed $\gamma > 0$,
 and so we only need to bound $h_{n+1}(R_n)$ and each $h_{n+1}(S_{n,j})$.
 Bounding $h_{n+1}(R_n)$ is simple using Proposition \ref{thm:distant-points},
 as for any $\theta \in R_n$,
 we have
 \[|\Phi_n'(e^{\sigma + i\theta})|
 \leq
 A^n \beta^{n/2}
 \left(
  \frac{
   L
  }{
   32 \beta
  }
 \right)^{-\frac{1}{2}(2^n - 1)} \ll \cparam^4 Z_n,\]
 and so $h_{n+1}(R_n) = o(\cparam^4)$.
 Finally, we will bound $h_{n+1}(S_{n,j})$.
 Using the bounds from Proposition \ref{thm:pf-bound}
 and Corollary \ref{thm:old-bp-probability-bound}, we have
 \begin{align*}
  h_{n+1}(S_{n,j})
  &\asymp
  \frac{1}{Z_n}
  \int_{-L_j^n}^{L_j^n}
   |\Phi_n'( \bp_j^n e^{\sigma + i \varphi} )|^{\etasize}
  \,\d\varphi\\
  &\leq
  \frac{
   B_{\etasize}^{n} \frac{
    \cparam^{\etasize( \frac{n-j}{4} + 1 )}
   }
   {
    \cparam^{\etasize 2^{n-j}}
   }
   \cparam^{\frac{\etasize}{2}( 1 - 2^{-j} )}
   \sigma^{-\left[ \etasize(1 - 2^{-j}) - 1 \right]}
  }{
   A^n
   \cparam^{\frac{\etasize}{2}(1 - 2^{-n})}
   \sigma^{-[\etasize(1 - 2^{-n}) - 1]}
  }\\
  &=
  \left( \frac{B_{\etasize}}{A} \right)^{\mspace{-5mu}n}
  \underbrace{
   \cparam^{\etasize( \frac{n-j}{4} + 1 )}
   \cparam^{-\frac{\etasize}{2}( 2^{-j} - 2^{-n} )}
  }_{
   o(\cparam^5)
  }
  \cparam^{-\etasize 2^{n-j}}
  \sigma^{\etasize( 2^{-j} - 2^{-n} )}\\
  &\ll
  \cparam^5
  \left( \frac{B_{\etasize}}{A} \right)^{\mspace{-5mu}n}
  \cparam^{-\etasize 2^{n-j}}
  \sigma^{\etasize 2^{-n}},
 \end{align*}
 then as $\sigma \leq \cparam^{2^{2^{1/\cparam}}}$,
 we have
 $\cparam^{-\etasize 2^{n-j}} \sigma^{\etasize 2^{-n}}
 \leq
 \cparam^{\etasize \left( 2^{2^{1/\cparam} - n} - 2^{n-j} \right)}
 \leq
 \cparam^{\etasize \left( 2^{2^{1/\cparam} - N} - 2^{N} \right)}
 $
 which decays faster than exponentially in $N$.
 Therefore $h_{n+1}(S_{n,j}) = o_T(\cparam^5)$,
 and so $\sum_{j=1}^{n-1} h_{n+1}(S_{n,j}) = o_T(\cparam^4)$,
 establishing \eqref{eq:concentration-result}.
 The second bound, \eqref{eq:symmetry-result},
 comes immediately from Corollary \ref{thm:symmetry-integrated}.
\end{proof}

\begin{remark}
 We have now seen that $(\theta^\cparam_{n})_{n \leq \rdown{T/\cparam}}$
 is very close to a simple symmetric random walk
 with step length $\beta \sim 2\cparam^{1/2}$,
 and so we expect $(\xi^\cparam_t)_{t \in [0,T]}
 = (\theta_{\rdown{t/\cparam}})_{t \in [0,T]}$
 will converge in distribution to $(2B_t)_{t \in [0,T]}$,
 where $B$ is a standard Brownian motion.
 We now state a result by McLeish \cite{mcleish-near-mg}
 which gives conditions for
 near-martingales to converge to a diffusive limit.
\end{remark}

\begin{corollary}[Corollary 3.8 of \cite{mcleish-near-mg}]
\label{cor:mcleish}
	Let $(X_{n,i})_{n, i \in \N}$ be an array of random variables,
	$J = [0,T]$ for $T > 0$ or $[0,\infty)$,
	and $(k_n)_{n \in \N}$ a sequence of right-continuous functions $J \to \N \cup \{0\}$.
  Write $W_n(t) = \sum_{i=1}^{k_n(t)} X_{n,i}$ for $t \in J$,
  and assume the following three limits hold in probability
  as $n \to \infty$:
	\begin{align}
	\label{mcleish-original-i}
		\sum_{j=1}^{k_n(t)}
			\E\left[
				X_{n,j}^2
				1[ |X_{n,j}| > \eps ]
				|
				X_{n,1}, \dotsc, X_{n,j-1}
			\right]
			&\to 0
			\text{ for all }
			\eps > 0,
		\\
	\label{mcleish-original-ii}
		\sum_{j=1}^{k_n(t)}
			\E\left[
				X_{n,j}^2
				|
				X_{n,1}, \dotsc, X_{n,j-1}
			\right]
			&\to t,
		\\
	\label{mcleish-original-iii}
		\sum_{j=1}^{k_n(t)}
			\left|\E\left[
				X_{n,j}
				|
				X_{n,1}, \dotsc, X_{n,j-1}
			\right]\right|
			&\to 0,
	\end{align}
	for all $t \in J$.
	Then $W_n \to B$ weakly in $D(J)$ as $n \to \infty$,
	where $B$ is a standard Brownian motion.

\end{corollary}

\begin{proof}[Proof of Proposition \ref{thm:driver-limit}]
	The bound $\P[ \deadtime \leq \rdown{T/\cparam} ] = O_T(\cparam^3)$
	is obtained immediately from Theorem \ref{thm:step-convergence},
	by observing for $1 \leq j \leq \rdown{T/\cparam}$ that
	\begin{align*}
		\P[ \deadtime \leq j ] \leq A\cparam^4
		+ \P[\deadtime \leq j-1].
	\end{align*}
	For the convergence of the driving function,
	we will apply Corollary \ref{cor:mcleish},
	replacing $n \to \infty$ by $\cparam \to 0$
	(this can be justified by showing the limit holds
	for any sequence of capacities $\cparam_n$
	tending to zero as $n \to \infty$)
	and $k_n(t)$ by $\rdown{t/\cparam}$.
	Then $X_{\cparam, j} = \theta_j - \theta_{j-1}$.
	Note that we will have $4t$ rather than $t$ as the limit
	in \eqref{mcleish-original-ii}, corresponding to
	a limit of $2B$ instead of $B$.\\
	
	The expectation of the $j$th term in \eqref{mcleish-original-i} is
	\begin{align*}
		\E\int_{-\pi}^{\pi}
			\varphi^2 h_{j}(\theta_{j-1} + \varphi)
			1[|\varphi| > \eps]
		\,\mathrm{d}\varphi
		&\leq
		\pi^2
		\E(\P( |\theta_{j} - \theta_{j-1}| > \eps \,|\,\theta_1, \dotsc, \theta_{j-1} ))\\
		&\leq
		\pi^2
		\P( \deadtime \leq j )
	\end{align*}
	when $\cparam$ is sufficiently small so $\beta + D < \eps$.
	Using our bound on
	$\P[\deadtime \leq \rdown{T/\cparam}]$,
	we see \eqref{mcleish-original-i}
	tends to zero in $L^1$
	and hence also in probability.
	
	Next, since $h_j$ approximates
	$\frac{1}{2}(\delta_{\theta_{j-1} - \beta} + \delta_{\theta_{j-1} + \beta})$,
	the $j$th term in \eqref{mcleish-original-ii} is
	\begin{align*}
		\int_{-\pi}^{\pi}
			\varphi^2 h_j(\theta_{j-1} + \varphi)
		\,\d \varphi
		&=
		\int_{\beta-D}^{\beta+D} \varphi^2 h_{j}(\theta_{j-1} + \varphi) \,\d\varphi
    +
  	\int_{-\beta-D}^{-\beta+D} \varphi^2 h_{j}(\theta_{j-1} + \varphi) \,\d\varphi
 	 	+
  	E_j\\
  	&=
  	(\beta + O(D))^2
  	\int_{\T \setminus F_{j-1}}
  		h_j(\theta)
  	\,\d \theta
  	+ E_j\\
  	&=
  	\beta^2 + O(\beta D) + E_j'
	\end{align*}
	where $E_j'$ is the sum of two terms:
	\begin{align*}
	\int_{F_{j-1}} \theta^2 h_j(\theta)\,\d \theta 
	&\leq \pi^2 1[\deadtime \leq \rdown{t/\cparam}] + \pi^2 A\cparam^4
	\intertext{and}
	(\beta^2 + O(\beta D))
	\int_{F_{j-1}} h_j(\theta) \,\d\theta
	&\leq 2\beta^2 1[\deadtime \leq \rdown{t/\cparam}] + 4 A \cparam^5
	\end{align*}
	(both bounds come from Theorem \ref{thm:step-convergence}).
	Hence \eqref{mcleish-original-ii} is
	\begin{align*}
		\sum_{j=1}^{\rdown{t/\cparam}}
			\int_{-\pi}^{\pi}
				\varphi^2 h_j(\theta_{j-1} + \varphi)
			\,\d \varphi
		&=
		\rdown{t/\cparam} \beta^2
		+ O\left(\frac{\beta D}{\cparam}\right)
		+ \sum_{j=1}^{\rdown{t/\cparam}} E_j'.
	\end{align*}
	Then
	\begin{align*}
		\E\left[
			\sum_{j=1}^{\rdown{t/\cparam}}
				|E_j'|
		\right]
		&\leq
		\rdown{\frac{t}{\cparam}} \left( (\pi^2 + 2\beta^2)\P[ \deadtime \leq \rdown{t/\cparam} ]
		+ \pi^2 A \cparam^4 + 4 A \cparam^5
		\right)\\
		&=
		O_T(\cparam^2),
	\end{align*}
	so \eqref{mcleish-original-ii} converges
	in $L^1$ to
	$\lim_{\cparam \to 0}(\rdown{t/\cparam}\beta^2) = 4t$ for any $t \in [0,T]$
	as $\cparam \to 0$.
	
	Finally, for the symmetry condition we can combine
	\eqref{eq:concentration-result} and \eqref{eq:symmetry-result}
	to bound the $j$th term in \eqref{mcleish-original-iii}:
	\begin{align*}
		\left| \int_{-\pi}^{\pi}
   		\varphi h_{j}(\theta_{j-1} + \varphi)
  	\,\d\varphi \right|
  	&\leq
  	\left|
  		\int_{\beta-D}^{\beta+D} \varphi
  		(
  			h_{j}(\theta_{j-1} + \varphi)
  			-
  			h_{j}(\theta_{j-1} - \varphi)
			)  			
  		\,\d\varphi
  	\right|
  	+
 	  \pi h_{j}(F_{j-1})\\
 	  &\leq
 	  \pi 1[\deadtime \leq \rdown{t/\cparam}]
 	  +
 	  (\beta + O(D)) A\cparam^{11/4}
 	  +
 	  A\cparam^4,
	\end{align*}
	so as with \eqref{mcleish-original-i},
	taking expectations it is simple to show that
	\eqref{mcleish-original-iii}
	tends to zero in $L^1$ and hence in probability
	as $\cparam \to 0$.
\end{proof}

\section{Alternative particle shapes}
\label{sec:alt-particles}

We believe that the results obtained above when using particles
of the form $(1, 1+d]$ can be extended to a more general family of particles.
In this case, depending on the form of the particles chosen,
we believe an SLE$_\kappa$
cluster can be obtained as the limit of an ALE$(0,\eta)$
for $\eta < -2$
for any $\kappa \in [4,\infty]$
(where SLE$_{\infty}$ is the growing disc
$t \mapsto e^t \overline{\D}$).

We will present below a few definitions
and statements to make this conjecture precise,
and some sketch arguments 
to support our claims.

\begin{definition}
	Let $\mathcal{P}$ be a family of subsets of $\Delta$,
	with $P \in \mathcal{P}$ if and only if:
	\begin{enumerate}[label=(\roman*)]
		\item $P \cup \overline{\D}$ is closed and bounded,
		\item for all $z \in P$, we have $z^* \in P$,
		\item $\overline{P} \cap \overline{\D} = \{ 1 \}$, and
		\item $P$ is convex.
	\end{enumerate}
	
	Note that for every $P \in \mathcal{P}$,
	there is a unique map $f^P : \Delta \to \Delta \setminus P$
	of the form $f^P(z) = e^\cparam z + O(1)$ near $\infty$
	for some $\cparam = \cparam(P) > 0$.
	As with the case $P = (1, 1+d]$
	there is also a unique $0 < \beta(P) < \pi$
	such that $f^P(e^{\pm i\beta(P)}) = 1$.
\end{definition}

Condition (iii) is necessary to obtain an SLE scaling result.
If the particle has a non-trivial base,
then the basepoints no longer sit in
increasingly deep ``fjords''
of low harmonic measure,
so the most recent basepoints
are no longer signficiantly more attractive
than the older basepoints.

Condition (iv) ensures the basepoints of each particle
are the areas of lowest harmonic measure.
For example the particle
$P_{\theta,\ell} = (1, 1 + e^{i\theta}\ell] \cup (1, 1 + e^{-i\theta}\ell]$
satisfies (i), (ii) and (iii),
but $(f^P)'$
has an additional singularity at $1$
as well as at $e^{\pm i \beta}$
if $0 < \theta < \pi$.
For certain values of $\theta$ the singularity at $1$
is in fact stronger than those at $e^{\pm i \beta}$.\\

Aside from particles of the form $(1, 1+d]$,
examples of particles in this family are discs $D_r$
of radius $r > 0$ and centre $1 + r$,
and line segments tangent to $\T$,
of the form
$T_{\ell} = [1 - i\ell, 1 + i\ell]$
for $\ell > 0$.

\begin{definition}
	Given a family $(P_\cparam)_{\cparam > 0}$
	of particles from $\mathcal{P}$,
	indexed by capacity so that $\cparam(P_\cparam) = \cparam$,
	we will call the family $\kappa$-\emph{stable}
	for $\kappa \in [0, \infty]$
	if $\beta(P_\cparam)^2 / \cparam \to \kappa$
	as $\cparam \to 0$.
\end{definition}

We can compute the maps $f^{D_r}$ and $f^{T_\ell}$
by elementary methods,
and so establish that both families are stable
and compute their respective $\kappa$s.
We write both maps here so that the reader
can satisfy themselves that they
have the same important properties as the map
$f^{(1,1+d]}$.

For $r > 0$ we have $\beta_r = \frac{\pi r}{1+r}$ and define
$m_r : \Delta \to \H$ by
\begin{align}
	m_r(z) = e^{i \beta_r}
	\frac{z - e^{-i \beta_r}}{z - e^{i \beta_r}},
\end{align}
and $\phi_r : \H \to \Delta \setminus D_r$ by
\begin{align}
	\psi_r(w) = \frac{\log w + i\beta}{\log w - i \beta},
\end{align}
where the logarithm is defined by $0 < \arg w < \pi$.
Then we have $f^{D_r} \colon \Delta \to \Delta \setminus D_r$
given by $f^{D_r} = \psi_r \circ m_r$.
It is then relatively easy to compute that the capacity of $D_r$,
$\cparam( D_r ) \sim \frac{1}{6} \pi^2 r^2$
and so (suitably reparameterised),
$(f^{D_r})_{r > 0}$ is $6$-stable.

The map for $T_\ell$ is somewhat more complicated.
Following the Schwarz--Christoffel computations in
\cite{tangent-solutions} (adapted for a symmetric tangent),
the calculations give rise to two quantities
as $\ell \to 0$:
$e_\ell \sim \ell$ (closely related to $\beta_{T_\ell}$)
and
$y_\ell = 2 - \frac{1}{6\pi} e_\ell^3 + o(e_{\ell}^3)$
(related to the capacity).
Using these,
we can define maps
$m_\ell : \Delta \to \H$,
$\psi_{\ell} : \H \to \H \setminus (\text{two arcs})$,
and
$\varphi_{\ell} : \H \setminus (\text{two arcs}) \to \Delta \setminus T_\ell$,
given by
\begin{align}
	m_\ell(z) &= i y_{\ell} \frac{z - 1}{z + 1},\\
	\psi_\ell(w) &= \frac{1}{2\pi} \log \left(\frac{w - e_\ell}{w + e_\ell}\right)
	- \frac{1 - e_\ell / \pi}{w},\\
	\varphi(\zeta)
	&=
	\frac{2\zeta + i}{2\zeta - i}.
\end{align}
Then $f^{T_\ell} = \varphi \circ \psi_\ell \circ m_\ell$.
Some calculations then give
$\beta_\ell^2 / \cparam(T_\ell) \sim \frac{12\pi}{\ell}$
as $\ell \to 0$,
so (again reparameterised by capacity),
$(T_\ell)_{\ell > 0}$
is $\infty$-stable.

\begin{figure}[h]
	\centering
	\begin{tikzpicture}
		\draw node (big-tangents) at (-1,0) {
			\includegraphics[width=0.30\textwidth,
											 trim = 0 0 0 0,
											 clip]{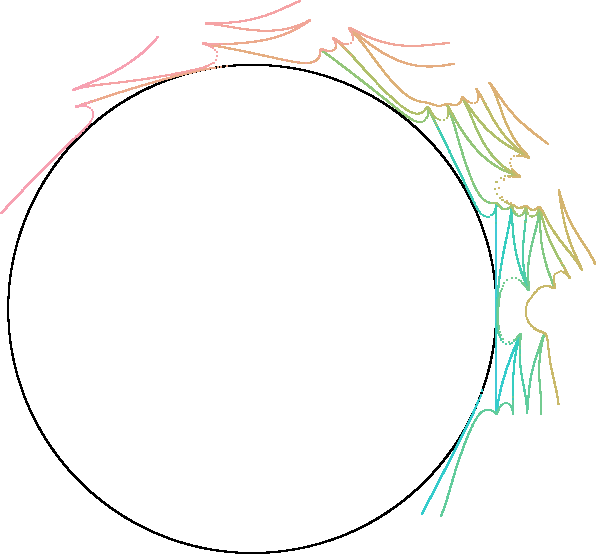}
		};
		\draw node (small-tangents) at (6.5,0) {
			\includegraphics[width=0.30\textwidth,
											 trim = 0 0 0 0,
											 clip]{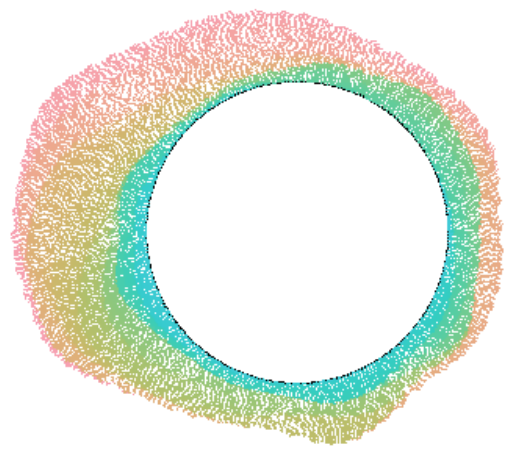}
		};
		\draw node (big-discs) at (0,-5) {
			\includegraphics[width=0.3\textwidth,
											 trim = 110 0 0 100,
											 clip,
											 angle=90]{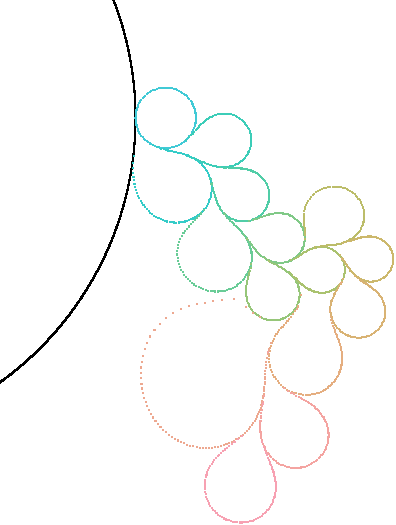}

		};
		\draw node (small-discs) at (6.5,-5) {
			\includegraphics[width=0.3\textwidth,
											 trim =130 0 0 80,
											 clip,
											 angle=90]{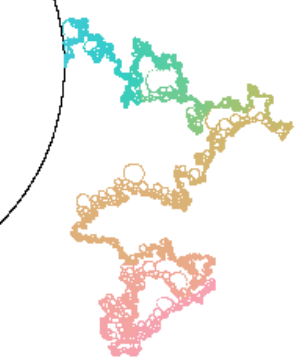}
		};	
	\end{tikzpicture}
	\caption{
	\label{fig:tangent}
		Clusters composed of
		tangent particles $T_\ell$ (top)
		and
		disc particles $D_r$ (bottom),
		generated with an angle sequence $\theta_k = \beta X_k$,
		for a simple symmetric random walk $X_k$,
		coloured according to the order of attachment
		(the earliest particles in blue and the latest in red).
    Note that these are \emph{not} simulations of an ALE process,
		but illustrations of what we conjecture their behaviour to be.
		For the tangent and disc particles (and even for the slit),
		the $\sigma$ necessary for convergence to an SLE
		is far too small
		to make simulating ALE practical
		in the regime this paper considers.
		The clusters on the right have 8,000 particles each
		and a total capacity around 0.2.
		The bottom-right cluster is close to an SLE$_6$,
		and the top-right cluster approximates
		an SLE$_\kappa$ with $\kappa$ around 377.
}
\end{figure}

Our main conjecture is that we have a version of
Proposition \ref{thm:driver-limit}
for every family of $\kappa$-stable particles,
and so the resulting cluster converges
in distribution
to an SLE$_{\kappa}$.

To grow most of the particles in $\mathcal{P}$
it is necessary to use Loewner's equation \eqref{eq:loewner} with
a driving \emph{measure} on $\T$ rather than a driving function.
We will not go into detail of this here,
but refer the reader to \cite{lawler-conformal-book}.
For a given particle $P$ with capacity $\cparam$,
we denote the driving probability measure
(evolving in time) by
$(\mu_t^P)_{0 \leq t \leq \cparam}$.

\begin{conjecture}
\label{driver-conjecture}
	Fix $T > 0$ and let $\eta < -2$.
	Suppose $(P_\cparam)_{\cparam > 0}$
	is a $\kappa$-stable family of particles
	from $\mathcal{P}$
	for $\kappa \in [4, \infty]$.
	Let $(\theta_n^\cparam)_{n \geq 1}$
	be the sequence of angles we obtain from the
	$\mathrm{ALE}(0,\eta)$
	process using particle $P_\cparam$
	and let $\sigma \leq c_0(P_{\cparam})$,
	some function which decays quickly as $\cparam \to 0$.
	
	Let $\deadtime
	= \inf\{ n \geq 2 : \min_{\pm}
	           |
	         \theta_n - (\theta_{n-1} \pm \beta_\cparam) > D \}$,
	where $D$ is a suitable function of $\sigma$ and $\cparam$.
	
	As $\cparam \to 0$,
	\[
		\P[ \deadtime \leq \rdown{T/\cparam} ] = O(\cparam^{\gamma})
	\]
	for some $\gamma > 1$.
	
	The driving measure for the whole cluster is
	$\d\xi_t^\cparam(\varphi)
	=
	\d\mu^{P_\cparam}_{t - \cparam\rdown{t/\cparam}}(
		\theta_{\rdown{t/\cparam}+1} + \varphi
	)$
	for $0 \leq t \leq T$.
	Then if $\kappa < \infty$,
	\[
		(\xi_t^{\cparam})_{t \in [0,T]}
		\to
		(\delta_{\sqrt{\kappa} B_t})_{t \in [0,T]}
		\text{ in distribution as } \cparam \to 0,
	\]
	as a random variable
	in the space of finite measures on $S = \T \times [0,T]$
	(equipped with the Wasserstein metric),
	and if $\kappa = \infty$
	then
	$(\xi_t^{\cparam})_{t \in [0,T]}$
	converges in the same sense
	to Lebesgue measure $\frac{1}{2\pi} \, \d\varphi \, \d t$
	on $S$.
\end{conjecture}


\begin{conjecture}[Generalisation of Theorem \ref{thm:main-result},
simple corollary of Conjecture \ref{driver-conjecture}]
	For $\eta, \sigma, \kappa$ and $(P_\cparam)_{\cparam > 0}$
	as in Conjecture \ref{driver-conjecture},
	let the $\mathrm{ALE}(0,\eta)$ cluster with $N = \rdown{T/\cparam}$
	particles of capacity $\cparam$ be $K_N^\cparam$.
	As $\cparam \to 0$, if $\kappa < \infty$
	then
	$K_N^\cparam$ converges in distribution
	as a random variable in $\mathcal{K}$
	to a radial SLE$_\kappa$ cluster of capacity $T$.
	If $\kappa = \infty$
	then $K_N^\cparam$ converges in $\mathcal{K}$
	to the disc $e^{T} \overline{\D}$.
\end{conjecture}

We believe the proof of Conjecture \ref{driver-conjecture}
is fairly straightforward for particles
where the map $f^{P_\cparam}$
is known explicitly,
such as $T_\ell$ and $D_r$.
As the support of $\mu_{t}^{P_\cparam}$
is $o(1)$ as $\cparam \to 0$,
proving convergence of the driving measure
is reduced to proving the angle sequence
approximates a symmetric random walk.
This follows quite simply if we can prove similar
bounds to those in
Theorem \ref{thm:step-convergence},
which we believe is simply a matter of
carefully verifying the type of explicit calculations
we were able to do for
$f^{(1,1+d]}$.

A proof for general $\kappa$-stable families
will require more generalised estimates of the maps
and their derivatives for particles in the class $\mathcal{P}$,
which we have not currently developed.

\begin{remark}
	One question which naturally arises
	is the significance of the $\kappa = 4$
	appearing in Theorem \ref{thm:main-result}
	for the slit particle.
	In fact we strongly believe
	that this is the minimal attainable $\kappa$
	for our ALE$(0,\eta < -2)$ models.
	Geometrically,
	slits $(1,1+d]$ are the only particles
	with ``zero width'',
	and $\kappa = 4$ marks a phase transition
	for SLE,
	since SLE$_4$ is a simple curve,
	and SLE$_\kappa$ for $\kappa > 4$
	is never a simple curve.
\end{remark}

\begin{proposition}
	For $0 \leq \kappa < 4$ there is no family of 
	$\kappa$-stable particles in $\mathcal{P}$.
\end{proposition}
\begin{proof}[Proof idea]
	First note that the family of slit particles
	$(Q_\cparam)_{\cparam > 0}
	= ((1,1+d(\cparam)])_{\cparam > 0}$
	is $4$-stable.
	For any particle $P \in \mathcal{P}$,
	we can express $(f^P)^{-1}$
	as the solution to the ``reverse'' Loewner equation
	with a symmetric driving measure,
	and then $e^{i\beta_P} = \lim_{\eps \downarrow 0} (f^P)^{-1}(e^{i\eps})$.
	An explicit calculation shows that
	if $P$ has capacity $\cparam$
	then $\beta_P \geq \beta_{Q_\cparam}$.
\end{proof}

\begin{remark}
	We are confident that an SLE$_\kappa$ can be realised
	as the limit of an ALE($0,\eta$) model
	for every $\kappa \in [4,\infty]$.
	For example, isoceles triangular particles joined
	to the circle at the apex, with vertex angle $\theta$,
	can interpolate between the slit particle
	$(1, 1+d]$ (the $\theta \to 0$ limit)
	and the tangent $T_\ell$ the $\theta \to \pi$ limit).
	We can therefore interpolate
	between $\kappa = 4$ and $\kappa = \infty$,
	realising every value in $(4,\infty)$
	as $\theta$ varies in $(0,\pi)$.
\end{remark}

\section*{Acknowledgements}
The author would like to thank Amanda Turner
for her guidance throughout the project,
and Vincent Beffara for his very useful comments
on an early version of the paper
about the $\eta = -\infty$ case.
I would also like to thank two anonymous referees
for detailed and helpful comments on the exposition,
and for the questions of one which led to
the inclusion of
\autoref{sec:alt-particles}.

\end{document}